\documentclass[11pt]{amsart}

\usepackage{amsmath}
\usepackage{amsthm}
\usepackage{amssymb}
\usepackage{hyperref}
\usepackage[margin=1.3in]{geometry}

\numberwithin{equation}{section}

\newtheorem{prop}{Proposition}
\newtheorem{lemma}[prop]{Lemma}
\newtheorem{thm}[prop]{Theorem}
\newtheorem{cor}[prop]{Corollary}

\numberwithin{prop}{section}
\newtheorem*{ack}{Acknowledgements}

\theoremstyle{definition}
\newtheorem{defn}[prop]{Definition}

\newtheorem{rmk}[prop]{Remark}

\newcommand{\dt}{\frac{\partial}{\partial t}}
\newcommand{\da}{\frac{\partial}{\partial a}}

\newcommand{\FF}{\mathcal F}
\renewcommand{\bar}[1]{\overline{#1}}
\newcommand{\del}{\partial}

\newcommand{\bi}{\bar{i}}
\newcommand{\bj}{\bar{j}}
\newcommand{\bk}{\bar{k}}
\newcommand{\bl}{\bar{l}}

\newcommand{\bp}{\bar{p}}

\newcommand{\til}[1]{\widetilde{#1}}

\newcommand{\KK}{\mathcal K}

\newcommand{\GG}{\mathcal G}
\newcommand{\LL}{\mathcal L}
\newcommand{\HH}{\mathcal H}

\DeclareMathOperator{\Ric}{Ric}
\DeclareMathOperator{\Rm}{Rm}

\DeclareMathOperator{\Ker}{Ker}

\DeclareMathOperator{\End}{End}

\begin{document}
	\title{Stability of the Almost Hermitian Curvature Flow}
	\author{Daniel J. Smith$^*$}
	\thanks{$^*$Supported in part by Research Training in Geometry and Topology grant DMS  0739208}
	
	\maketitle
	
\begin{abstract}
The Almost Hermitian Curvature flow was introduced in \cite{SCF} by Streets and Tian in order to study almost hermitian structures, with a particular interest in symplectic structures. This flow is given by a diffusion-reaction equation. Hence it is natural to ask the following: which almost hermitian structures are dynamically stable? An almost hermitian structure $(\til{\omega},\til{J})$ is dynamically stable if it is a fixed point of the flow and there exists a neighborhood $\mathcal{N}$ of $(\til{\omega},\til{J})$ such that for any almost hermitian structure $(\omega(0),J(0)) \in \mathcal{N}$ the solution of the Almost Hermitian Curvature flow starting at $(\omega(0),J(0))$ exists for all time and converges to a fixed point of the flow. We prove that on a closed K\"{a}hler-Einstein manifold $(M,\til{\omega},\til{J})$ such that either $c_1(\til{J}) <0$ or $(M,\til{\omega},\til{J})$ is a Calabi-Yau manifold, then the K\"{a}hler-Einstein structure $(\til{\omega},\til{J})$ is dynamically stable.
\end{abstract}

\section{Introduction}
Let $(M, J, g)$ be a closed almost complex manifold such that $J$ is compatible with the Riemannian metric $g$, that is for any vector fields $X$ and $Y$ we have $g(X,Y) = g(J(X),J(Y))$. To the metric $g$ we associate the 2-form $\omega$ defined by $\omega(X,Y) = g(J(X),Y)$. We call such a pair $(\omega, J)$ an \emph{almost hermitian structure}.\\ 
\indent The Ricci flow has proven to be a successful tool in studying the Riemannian geometry of manifolds. Therefore it is natural to attempt to use a parabolic flow to understand the almost hermitian geometry of almost complex manifolds. However, the Ricci flow does not, in general, preserve the set of almost hermitian structures. In \cite{SCF}, Streets and Tian introduce the Almost Hermitian Curvature flow (AHCF), which is a weakly-parabolic flow on the space of almost hermitian structures.\\
\indent AHCF generalizes K\"{a}hler Ricci flow in the sense that if the initial structure $(\omega_0,J_0)$ is K\"{a}hler, then the evolution of $(\omega(t),J(t))$ by AHCF coincides with K\"{a}hler Ricci flow. In \cite{HCF}, Streets and Tian construct a parabolic flow on the space of hermitian structures $(\omega,J)$ (here $J$ is integrable), called Hermitian Curvature flow (HCF). AHCF also generalizes HCF. In addition to HCF, Gill has also introduced a parabolic flow of hermitian structures called Chern-Ricci flow (see \cite{Gill1}, \cite{CRF1}, \cite{CRF2}).\\
\indent As we will see below AHCF is, in fact, a family of geometric flows. Streets and Tian have a particular interest in one of these flows, called Symplectic Curvature flow (SCF). Given an almost hermitian structure $(\omega_0,J_0)$ such that $d\omega_0 = 0$, under SCF $\omega(t)$ is a closed form as long as the flow exists. Therefore, SCF is a tool which can be used to study symplectic structures. Hence we see that AHCF is a very general family of geometric flows.\\
\indent The Almost Hermitian Curvature flow is a coupled flow of metrics and almost complex structures. It is written 
\begin{gather} \label{AHCF}
\begin{split}
	\dt \omega &=-2S+H+Q\\ 
	\dt J &=-\KK+\HH.
\end{split}
\end{gather}
 $S$ is a ``Ricci-type" curvature. In particular, $S_{ij}=\omega^{kl}\Omega_{klij}$ and $\Omega$ is the curvature of the almost-Chern connection $\nabla$. That is, $\nabla$ is the unique connection satisfying $\nabla \omega =0$, $\nabla J =0$ and $T^{1,1} =0$. $T^{1,1}$ is the $(1,1)$ component of the torsion of $\nabla$. $Q$ is any $(1,1)$ form that is quadratic in the torsion of $\nabla$. $\KK_j^i= \omega^{kl}\nabla_kN_{lj}^i$ where $N$ is the Nijenhuis tensor with respect to $J$. $\HH$ is any endomorphism of $TM$ that is quadratic in $N$ and skew commutes with $J$. The term $H(X,Y) \dot{=} \frac{1}{2}\big[ \omega((-\KK + \HH)(X),J(Y))+\omega(J(X),(-\KK + \HH)(Y)) \big]$ is required in order to maintain the compatibility of $\omega_t$ with $J_t$. Streets and Tian prove short-time existence and uniqueness (see Theorem 1.1 in \cite{SCF}) of the flow starting at an almost hermitian structure $(\omega(0),J(0))$. Notice that the generality with which the tensors $Q$ and $\HH$ are defined implies that (\ref{AHCF}) is in fact a family of geometric flows. This family of geometric flows includes Hermitian Curvature flow, Symplectic Curvature flow and K\"{a}hler Ricci flow. Associated to AHCF is the volume-normalized version of the flow (VNAHCF), the volume-normalized version is the one with which we will work.\\
\indent One natural question to ask is: does $M$ admit a K\"{a}hler-Einstein structure? If so, is it detected by VNAHCF? The main result of the paper is the following:

\begin{thm} \label{dynamic stability} 
Let $(M^{2n},\til{\omega},\til{J})$ be a closed complex manifold with $(\til{\omega},\til{J})$ a K\"{a}hler-Einstein structure such that either $c_1(\til{J}) < 0$ or $(M,\til{\omega},\til{J})$ is a Calabi-Yau manifold. Then there exists $\epsilon >0$ such that if $(\omega(0),J(0))$ is an almost hermitian structure with $\big|(\omega(0)-\til{\omega},J(0)-\til{J})\big|_{C^{\infty}} < \epsilon$, then the solution to the volume normalized AHCF starting at $(\omega(0),J(0))$ exists for all time and converges exponentially to a K\"{a}hler-Einstein structure $(\omega_{KE},J_{KE})$.
\end{thm}

\begin{rmk} Theorem \ref{dynamic stability} gives evidence that the Almost Hermitian Curvature flow reflects the underlying almost hermitian geometry of $M$. 
\end{rmk}

\begin{rmk} In this paper we define a \emph{Calabi-Yau manifold} $(M,\til{\omega},\til{J})$ to be a compact K\"{a}hler manifold with trivial canonical bundle such that $\til{\omega}$ is a K\"{a}hler-Einstein metric with $\Ric (\til{\omega}) = 0$.
\end{rmk}

\begin{rmk} In the case when $c_1(\til{J}) < 0$, the K\"{a}hler-Einstein structure that the flow starts close to is the same one that the flow converges to, in other words $(\til{\omega},\til{J}) = (\omega_{KE},J_{KE})$. This is proved in Theorem \ref{dynamic stability c1 < 0}.\\
\indent In the Calabi-Yau case we cannot guarantee that $(\til{\omega},\til{J})$ and $(\omega_{KE},J_{KE})$ are the same Calabi-Yau structure.
\end{rmk}

The notion of stability in Theorem \ref{dynamic stability} is often referred to as dynamic stability. Dynamic stability has also been studied in the case of the Hermitian Curvature flow by Streets and Tian (\cite{HCF}) and for the Ricci flow by Sesum (\cite{Sesum1}) and by Guenther, Isenberg, and Knopf (\cite{GIK}).\\
\indent The first step in proving Theorem \ref{dynamic stability} is to show that K\"{a}hler-Einstein structures behave like sinks of the linear flow associated to VNAHCF, this is done in Section \ref{linstab}. Also in Section \ref{linstab}, we derive parabolic estimates for the VNAHCF (see Theorem \ref{start close, stay close}).

Next, in Section \ref{c1<0} we prove Theorem \ref{dynamic stability} in the case when $c_1(\til{J}) < 0$. Finally, in the last section we complete the proof of Theorem \ref{dynamic stability} by showing how to find a K\"{a}hler-Einstein structure $(\omega_{KE},J_{KE})$ to which the flow exponentially converges in the Calabi-Yau case.

\begin{ack}{\rm 
The author would like to thank his advisor Jon Wolfson for all of his help and suggestions on this project.}
\end{ack}

\section{Linear Stability and Parabolic Estimates} \label{linstab}
To prove theorem \ref{dynamic stability} we first show that, on a linear level, any K\"{a}hler-Einstein structure $(\til{\omega},\til{J})$ is a ``degenerate sink" with respect to the Almost Hermitian Curvature flow, meaning that the linear operator associated to the non-linear flow is negative semi-definite at K\"{a}hler-Einstein structures. For notation sake write VNAHCF:
\begin{gather}
\begin{split} \label{flow of w,J}
	\dt \omega &= \FF  \\
	\dt J &= \GG.
\end{split}
\end{gather}
To the operator $(\FF, \GG)$, we have the associated linear operator $(\dot{\FF},\dot{\GG})$. In particular, we consider a one-parameter family of compatible, unit volume, almost hermitian structures $(\omega(a), J(a))$ and $(\dot{\FF},\dot{\GG}) \doteq \da \big|_{a=0}(\FF, \GG)(\omega(a), J(a))$. Similarly we write $(\dot{\omega},\dot{J}) \doteq \da \big|_{a=0}(\omega(a), J(a))$. 

\defn An almost-hermitian structure $(\omega, J)$ is called \emph{static} provided $(\FF(\omega,J),\GG(\omega,J)) = 0$. Moreover, a static structure $({\omega},{J})$ is \emph{linearly stable} if the linearization $\LL \doteq (\dot{\FF},\dot{\GG})_{({\omega},{J})}$ is negative semi-definite, that is $\langle \LL \cdot, \cdot \rangle_{L^2({g})} \leq 0$.\\

\noindent Notice that K\"{a}hler-Einstein structures are static under VNAHCF. Next, we prove

\thm \label{linear stability} Let $(M, \til{J})$ be a closed complex manifold with $c_1(\til{J}) \leq 0$, then any K\"{a}hler-Einstein structure $(\til{\omega},\til{J})$ on $M$ is linearly stable.

\proof Employing the DeTurck trick as in Proposition 5.4 and 5.5 of \cite{SCF} the weak-ellipticity of $(\FF, \GG)$ follows. Furthermore, computing the linearization of $(\FF, \GG)$ at a K\"{a}hler-Einstein structure, in complex coordinates with respect to $\til{J}$, the linearization is written
\begin{align}
\dot{\FF}_{i\bj} &= -2\nabla^*\nabla \dot{\omega}_{i\bj} + 2\dot{\omega}^{k\bl}R_{k \bl i \bj}\label{F11}\\
\dot{\FF}_{i j} &= -2\nabla^*\nabla \dot{\omega}_{i j}\label{F20}\\
\dot{\GG}_{\bj}^i &= -2\nabla^*\nabla \dot{J}_{\bj}^i + 2\dot{J}^{pq}R_{\bj pq}^i\label{G}.
\end{align}
Here $\nabla^*$ is the $L^2(\til{g})$ adjoint of $\nabla$ and $R$ denotes the Riemannian curvature of $\til{g}$. \\

\indent To show that $(\til{\omega},\til{J})$ is linearly stable we have to deal with the fact that in (\ref{F11}) and (\ref{G}) the lower order terms do not have a sign. To see that the linearized operator is negative semi-definite at K\"{a}hler-Einstein structures we use a couple of Weitzenb\"{o}ck-Bochner formulas (cf. \cite{Besse}).

\begin{lemma}\label{bochner F} Let $\alpha$ and $\beta$ be a (0,2) and (1,1) form respectively. If $\til{g}$ is a K\"{a}hler-Einstein metric, then we have
\begin{align}
(\Delta_d\alpha)_{\bi \bj} &= 2\nabla^*\nabla \alpha_{\bi \bj} + 2\frac{s}{n}\alpha_{\bi \bj}\label{bochner20}\\
(\Delta_d\beta)_{i\bj} &= 2\nabla^*\nabla \beta_{i\bj}-2\beta^{k\bl}R_{k\bl i \bj}+2\frac{s}{n}\beta_{i\bj}\label{bochner11}.
\end{align}
\end{lemma}

\noindent Where $\Delta_d$ is the Hodge Laplacian with respect to $\til{g}$ and $s$ is the scalar curvature of $\til{g}$. In addition we will use another Weitzenb\"{o}ck-Bochner formula.

\begin{lemma}\label{bochner G} Given a $T^{1,0}(M,\til{\omega},\til{J})$ valued $(0,1)$-form, $\phi$ and K\"{a}hler-Einstein metric $\til{g}$ we have:
\begin{align}
({\Delta}_{\overline{\del}}\phi)_{\bj}^i=\nabla^*\nabla \phi_{\bj}^i-\phi^{pq}R_{\bj pq}^i + \frac{s}{n}\phi_{\bj}^i.\label{bochnerG}
\end{align}
Where ${\Delta}_{\overline{\del}}$ represents the complex laplacian $ \overline{\del}\overline{\del}_{\til{g}}^* +  \overline{\del}^*\overline{\del}_{\til{g}}$.
\end{lemma}

\noindent Therefore using equations (\ref{F11}), (\ref{F20}), (\ref{bochner20}) and (\ref{bochner11}), we have that
\begin{align}
	\dot{\FF} = -\Delta_d\dot{\omega}+2\frac{s}{n}\dot{\omega}.\label{linearization of F}
\end{align}
Similarly, using equations (\ref{G}) and (\ref{bochnerG}), we have
\begin{align}
	\dot{\GG} = -\Delta_{\overline{\del}}\dot{J} + 2 \frac{s}{n}\dot{J}.\label{linearization of G}
\end{align}
Combining (\ref{linearization of F}) and (\ref{linearization of G}) we see that
\begin{align}
	\LL (\dot{\omega},\dot{J}) = \left(-\Delta_d\dot{\omega}+2\frac{s}{n}\dot{\omega}, -\Delta_{\overline{\del}}\dot{J} + 2 		\frac{s}{n}\dot{J}\right). \label{LL neg def}
\end{align} 
Finally since $c_1(\til{J}) \leq 0$ implies that $s \leq 0$; by integrating the theorem follows.
\qed

\medskip
Notice that if $c_1(\til{J}) < 0$, then the scalar curvature of $\til{g}$ is negative; that is $s < 0$. Hence from (\ref{LL neg def}) it follows that if $c_1(\til{J}) < 0$, then $\LL$ is strictly negative definite with respect to $L^2(\til{g})$. Let $\lambda = \min \{ |\lambda_i| : \text{$\lambda_i$ is an eigenvalue of $\LL$} \}$. Further let $\mathcal{C}$ denote the space of almost hermitian structures modulo diffeomorphism. Therefore we have proved the following corollary.

\begin{cor}\label{c1<0 implies LL<0}
Let $(M,\til{\omega},\til{J})$ be a closed complex manifold such that $(\til{\omega},\til{J})$ is a K\"{a}hler-Einstein structure and moreover $c_1(\til{J}) < 0$. Let $\psi \in T_{(\til{\omega},\til{J})}\mathcal{C}$, then
$$\big\langle \LL_{(\til{\omega},\til{J})} \psi, \psi \big\rangle_{L^2(\til{g})} \leq -\lambda |\psi|_{L^2(\til{g})}^2.$$
\end{cor}
Corollary \ref{c1<0 implies LL<0} will be crucial to proving Theorem \ref{dynamic stability} in the case when $c_1(\til{J}) < 0$ (see Section \ref{c1<0}).\\
\indent Fix a K\"{a}hler-Einstein structure $(\til{\omega},\til{J})$ and let $(\omega(t),J(t))$ be a solution of the coupled system (\ref{flow of w,J}) starting at an initial almost hermitian structure $(\omega(0),J(0))$. We will quantify the amount by which the solution deviates from $(\til{\omega},\til{J})$ using 
\begin{align*}
	\rho(t) = (\omega(t)-\til{\omega},J(t)-\til{J}).
\end{align*} 
Notice that $\rho(t) \in \Lambda^2(M) \times \End (TM)$. Throughout the paper we use the operator norm on $\End (TM)$.\\
\indent As noted in the proof of Theorem 1.1 in \cite{SCF}, $\mathcal{C}$ is a non-linear manifold. In the following lemma we will see that $\rho(t) \notin T_{(\til{\omega},\til{J})}\mathcal{C}$, however $\rho(t)$ can be estimated by an element of the tangent space $T_{(\til{\omega},\til{J})}\mathcal{C}$.

\medskip
\lemma \label{psi is quadratic/cubic} Fix $t$ and let $(\omega(t),J(t))$ be an almost hermitian structure. Write $\omega(t) = \til{\omega} + h(t)$ and $J(t) = \til{J} + K(t)$, in other words $\rho(t) = (h(t),K(t))$. If $|\rho(t)|_{C^0} < 1$, then there exists $\psi(t) \in T_{(\til{\omega},\til{J})}\mathcal{C}$ so that
\begin{align}
	|\psi(t)|_{C^k} \leq |\rho(t)|_{C^k}, \label{psi leq rho}
\end{align}
\begin{align}
	|\rho(t)|_{L^2} \leq |\psi(t)|_{L^2} + C_1|\psi(t)|_{L^2}^2 \label{rho leq psi L2}
\end{align}
and
\begin{align}
	|\rho(t)|_{C^k} \leq |\psi(t)|_{C^k} + C_2|\psi(t)|_{C^k}^2 \label{rho leq psi}
\end{align}
where $C_1$ and $C_2$ depend on the $L^2$ and $C^k$ norms of $\rho(t)$ respectively.

\proof We will begin by studying the tangent space $T_{(\til{\omega},\til{J})}\mathcal{C}$. Let $(\omega_s,J_s)$ denote a path of almost hermitian structures such that $(\omega_s,J_s)|_{s=0}=(\til{\omega},\til{J})$ and let $\frac{\partial}{\partial s} \big|_{s=0}(\omega(s), J(s)) \doteq (\dot{\omega},\dot{J})$. Given vector fields $X$ and $Y$, the compatibility condition is written:
\begin{align*}
	\omega_s(X,Y) = \omega_s(J_s(X),J_s(Y))
\end{align*}
and the almost complex condition is written:
\begin{align*}
	J^2_s(X) = -X.
\end{align*}
Hence the linearized compatibility and almost complex conditions are given by:
\begin{align}
	\dot{\omega}(X,Y) &= \dot{\omega}(\til{J}(X),\til{J}(Y))+\til{\omega}(\dot{J}(X),\til{J}(Y))+\til{\omega}(\til{J}(X),\dot{J}(Y))	\label{linearized compatibility} \\
	0&=\dot{J} \circ \til{J}(X) +\til{J} \circ \dot{J}(X). \label{linearized almost complex}
\end{align}
From (\ref{linearized almost complex}) we see that the tangent space to the space of almost complex structures is given by endomorphisms that skew-commute with $\til{J}$. Equivalently, $\dot{J}$ can be viewed as a section of $\left[\Lambda^{0,1} \otimes T^{1,0}\right] \oplus \left[\Lambda^{1,0} \otimes T^{0,1}\right]$. Here we use $\til{J}$ to decompose $TM = T^{1,0}M \oplus T^{0,1}M$.\\
\indent First we will prove that the endomorphism $K(t)$ can be estimated by an element of the tangent space to the space of almost complex structures at $\til{J}$. For the sake of notation we will often write $K(t) = K$.\\ 
\indent Using $\til{J}$ we decompose $K = K_{0,1}^{1,0} + K_{0,1}^{0,1} + K_{1,0}^{1,0} + K_{1,0}^{0,1}$ where $K_{0,1}^{1,0} : T^{0,1} \rightarrow T^{0,1}$, equivalently
\begin{align*}
	K_{0,1}^{1,0} \in \Lambda^{0,1} \otimes T^{1,0}. 
\end{align*}

\indent Take $\psi(t) \in T_{(\til{\omega},\til{J})}\mathcal{C}$ and write $\psi(t) = (\psi_1(t),\psi_2(t)) \in \Lambda^2(M) \times \End (TM)$. We define 
\begin{align}
	\psi_2(t) \doteq K_{0,1}^{1,0} + K_{1,0}^{0,1}. \label{psi_2}
\end{align}
That is, $\psi_2(t)$ is defined to be the projection of $K$ onto $\left[\Lambda^{0,1} \otimes T^{1,0}\right] \oplus \left[\Lambda^{1,0} \otimes T^{0,1}\right]$. Next we will show that $K_{0,1}^{0,1}$ and $K_{1,0}^{1,0}$ are quadratic in $\psi_2(t)$. We will only prove this for $K_{0,1}^{0,1}$ since the same argument applies to $K_{1,0}^{1,0}$.\\
\indent Using that $J(t)$ is an almost complex structure we see that $K(t)$ satisfies:
\begin{align}
	0&=K \circ \til{J}(X) +\til{J} \circ K(X) + K^2(X). \label{K almost complex}
\end{align}
Now for $K$ acting on $T^{0,1}$ we will write $K = K_{0,1}^{0,1} + K_{0,1}^{1,0}$. Therefore using (\ref{K almost complex}), on $T^{0,1}$ we have
\begin{align*}
	0 &= -2\sqrt{-1}K_{0,1}^{0,1} + K_{0,1}^{0,1} \circ K_{0,1}^{0,1} + K_{0,1}^{1,0} \circ K_{0,1}^{0,1} + K_{1,0}^{1,0} \circ 	K_{0,1}^{1,0} + K_{1,0}^{0,1} \circ K_{0,1}^{1,0}
\end{align*}
and so by type consideration,
\begin{align}
	K_{0,1}^{0,1} = -\frac{\sqrt{-1}}{2}\left(K_{0,1}^{0,1} \circ K_{0,1}^{0,1} + K_{1,0}^{0,1} \circ K_{0,1}^{1,0}\right). 	\label{K0101}
\end{align}
Notice that on $T^{0,1}$, $K_{1,0}^{0,1} \circ K_{0,1}^{1,0} = \psi_2(t)^2$. Hence we are able to write $K_{0,1}^{0,1}$ in terms of $\left(K_{0,1}^{0,1}\right)^2$ and a term that is quadratic in $\psi_2(t)$.\\
\indent Next, consider the first term on the right-hand side of (\ref{K0101}), $\left(K_{0,1}^{0,1}\right)^2$. We will use (\ref{K0101}) to show that $\left(K_{0,1}^{0,1}\right)^2$ can be expressed as $\left(K_{0,1}^{0,1}\right)^4$ plus terms which are quadratic in $\psi_2(t)$. Plugging (\ref{K0101}) into each factor of $\left(K_{0,1}^{0,1}\right)^2$, we see that
\begin{align*}
	\left(K_{0,1}^{0,1}\right)^2 = -\frac{1}{4}\left[\left(K_{0,1}^{0,1}\right)^4 + \left(K_{0,1}^{0,1}\right)^2 \circ \psi_2^2 + 		\psi_2^2 \circ \left(K_{0,1}^{0,1}\right) + \psi_2^4 \right],
\end{align*}
which can be substituted into the term $K_{0,1}^{0,1} \circ K_{0,1}^{0,1}$ in (\ref{K0101}). Iterating this process by successively plugging (\ref{K0101}) into the highest power term in $K_{0,1}^{0,1}$, we see that $K_{0,1}^{0,1}$ can be expressed as a series. Notice that since $|\rho|_{C^0} < 1$ it follows that $\big|K_{0,1}^{0,1}\big|_{C^0} < 1$, and so this series converges. Therefore
\begin{align} 
	K_{0,1}^{0,1} = \psi_2(t)^2 +\text{ $[$higher-power terms in $\psi_2$ $\circ$ higher-power terms in $K]$.} \label{K0101 higher order in psi}
\end{align}
\indent Next we will show that the two form $h(t)$ can be estimated by an element of the tangent space to the space of compatible metrics. Notice that for vector fields $X$ and $Y$,
\begin{align*}
	\dot{\omega}(X,Y) - \dot{\omega}(\til{J}(X),\til{J}(Y)) = 2\dot{\omega}^{(2,0)+(0,2)}(X,Y),
\end{align*}
and so by (\ref{linearized compatibility})
\begin{align*}
	2\dot{\omega}^{(2,0)+(0,2)}(X,Y) = \til{\omega}(\dot{J}(X),\til{J}(Y))+\til{\omega}(\til{J}(X),\dot{J}(Y)).
\end{align*} 
Using the compatibility of $\omega(t)$ and $J(t)$, we see that $h(t) = \omega(t)-\til{\omega}$ and $K(t) = J(t)-\til{J}$ satisfy:
\begin{gather} \label{h compatibility}
\begin{split}
	h(X,Y) &= h(\til{J}(X),\til{J}(Y))+\til{\omega}(K(X),\til{J}(Y))+\til{\omega}(\til{J}(X),K(Y)) \\
	&\hspace{.152 in}+\til{\omega}(K(X),K(Y))+h(K(X),\til{J}(Y))+h(\til{J}(X),K(Y)) \\
	&\hspace{.152 in}+h(K(X),K(Y)).
\end{split}
\end{gather}
We define $\psi_1(t)$ as follows
\begin{align}
	\psi_1^{(1,1)} &\doteq h^{(1,1)} \label{psi111} \\
	\psi_1^{(2,0)+(0,2)}(X,Y) &\doteq \til{\omega}(K(X),\til{J}(Y)) + \til{\omega}(\til{J}(X),K(Y)) \label{psi20021} \\
	 &= \til{\omega}(\psi_2(X),\til{J}(Y)) + \til{\omega}(\til{J}(X),\psi_2(Y)). \label{psi20022}
\end{align}
The last equality follows from the definition of $\psi_2$ and the fact that $\til{\omega}$ is of type $(1,1)$.\\
\indent Next we will show that $h^{(2,0)+(0,2)} - \psi_1^{(2,0)+(0,2)}$ can be expressed as terms that are quadratic in $\psi(t)$. Combining (\ref{h compatibility}), (\ref{psi20021}) and (\ref{psi20022}) we have
\begin{gather} \label {higher-order K terms}
\begin{split}
	2\left(h^{(2,0)+(0,2)}(X,Y) - \psi_1^{(2,0)+(0,2)}(X,Y)\right) &= h(K(X),\til{J}(Y))+h(\til{J}(X),K(Y)) \\
	&\hspace{.152 in}\til{\omega}(K(X),K(Y)) +h(K(X),K(Y)).
\end{split}
\end{gather}
\indent As we proved above in (\ref{psi_2}) and (\ref{K0101 higher order in psi}), $K$ can be written in terms of $\psi_2$ and hence the terms in the second line of (\ref{higher-order K terms}) are higher-power in $\psi_2$. Next we consider the term $h(K(X),\til{J}(Y))$. Since the left-hand side of (\ref{higher-order K terms}) is a section of $\Lambda^{(2,0) + (0,2)}$, let $X, Y \in T^{0,1}M$. So for $X, Y \in T^{0,1}M$ we can write the components of $h(K(X),\til{J}(Y))$ as
\begin{align}
	K_{0,1}^{0,1}h^{(0,2)} + K_{0,1}^{1,0}h^{(1,1)}. \label{K0110h02}
\end{align}
\indent From (\ref{psi_2}) and (\ref{psi111}) we see that the second term in (\ref{K0110h02}) is quadratic in $\psi$. By (\ref{K0101 higher order in psi}) the first term is quadratic in $\psi$ plus higher-power terms in $\psi$ composed with higher-power terms in $\rho$. Notice that the same argument can be applied to $h(\til{J}(X),K(Y))$. Abusing notation we let $\psi * \psi$ denote terms which are quadratic in $\psi$ plus terms that are higher-power in $\psi$ composed with terms that are higher-power in $\rho$. Therefore we have
\begin{align}
	h^{(2,0)+(0,2)}(X,Y) - \psi_1^{(2,0)+(0,2)}(X,Y) &= \psi * \psi. \label{psi1 higher power}
\end{align}
\indent Notice that by the definition of $\psi(t)$, given in (\ref{psi_2}) (\ref{psi111}) and (\ref{psi20021}), the inequality
\begin{align*}
	|\psi(t)|_{C^k} \leq |\rho(t)|_{C^k}
\end{align*}
follows immediately. Again using the definition of $\psi(t)$ along with (\ref{K0101 higher order in psi}) and (\ref{psi1 higher power}) we see that
\begin{align*}
	|\rho(t)|_{C^k} \leq |\psi(t)|_{C^k} + C|\psi(t)|_{C^k}^2
\end{align*}
where $C$ depends on the $C^k$ norm of $\rho(t)$. Notice that (\ref{rho leq psi L2}) follows analogously.
\qed

\medskip
In Theorem \ref{linear stability} we proved that the linearization of $(\FF, \GG)$, denoted $\LL$, is negative semi-definite on $T_{(\til{\omega},\til{J})}\mathcal{C}$. The goal is to use the sign on $\LL$ to prove exponential decay of $\rho(t)$. However as we observed in the previous lemma, $\rho(t) \notin T_{(\til{\omega},\til{J})}\mathcal{C}$. To deal with this we will prove exponential decay of $\psi(t) \in T_{(\til{\omega},\til{J})}\mathcal{C}$ which, by (\ref{rho leq psi}), will prove exponential decay of $\rho(t)$.\\
\indent Next we show that $\psi(t)$ evolves by a parabolic flow equation and moreover that we have estimates on the non-linear part of the flow.
  
\lemma \label{the flow of psi} Let $\LL$ be the differential operator defined by (\ref{LL neg def}). Then for $\psi(t)  \in T_{(\til{\omega},\til{J})}\mathcal{C}$ defined by (\ref{psi_2}) (\ref{psi111}) and (\ref{psi20022}) we have
\begin{enumerate}
	\item $\dt \psi(t) = \LL (\psi(t)) + A((\til{\omega},\til{J}),\psi(t))$
	\item $|A((\til{\omega},\til{J}),\psi(t))|_{C^k} \leq C(|\psi(t)|_{C^k}|\nabla^2\psi(t)|_{C^k} + |\nabla \psi(t)|_{C^k}^2)$
\end{enumerate}
where $C$ depends on the $C^k$ norm of $\rho(t)$.
\proof 

To prove $\dt \psi(t) = \LL (\psi(t)) + A((\til{\omega},\til{J}),\psi(t))$ we first study the evolution of $\rho(t)$. Notice that since $(\til{\omega},\til{J})$ is independent of $t$,
\begin{align} 
	\dt \rho(t)=\dt (\omega(t),J(t))=\big( \FF (\omega(t),J(t)),\GG ( \omega(t),J(t))\big).\label{evolution rho}
 \end{align}
Furthermore since $(\til{\omega},\til{J})$ is a static structure, when we linearize $(\FF, \GG)$ at $(\til{\omega},\til{J})$ in the direction $\psi(t)$, we have
\begin{align} 
	(\FF, \GG)=\LL(\psi(t))+A((\til{\omega},\til{J}),\rho(t)).\label{linearization 2}
\end{align}
Hence from (\ref{evolution rho}) and (\ref{linearization 2}) it follows that
\begin{align} 
	\dt \rho(t)=\LL(\psi(t))+A((\til{\omega},\til{J}),\rho(t)), \label{linearization 3}
\end{align}
where $A$ represents the error in approximating $(\FF, \GG)$ by the linearization $\LL$. As in \cite{HCF} and \cite{Sesum1} we have the following $C^k$ bounds on $A$:
\begin{align*}
	|A|_{C^k}\leq C(|\rho|_{C^k}|\nabla^2\rho|_{C^k}+|\nabla \rho|_{C^k}^2). 
\end{align*}
Therefore by Lemma \ref{psi is quadratic/cubic} we have the following bounds on $A$:
\begin{align}
	|A|_{C^k}\leq C(|\psi|_{C^k}|\nabla^2\psi|_{C^k}+|\nabla \psi|_{C^k}^2). \label{A bounds 3}
\end{align}
\indent Next we will use the definition of $\psi(t)$ and the evolution of $\rho(t)$ to derive an evolution equation for $\psi(t)$. By the definition of $\psi_2(t)$ and the $(1,1)$ part of $\psi_1(t)$ (see (\ref{psi_2}) and (\ref{psi111}) respectively) we have
\begin{align}
	\dt \psi_2(t) &= \dt \left(K_{0,1}^{1,0} + K_{1,0}^{0,1} \right) \label{evolution psi_2}\\
	\dt \psi_1^{(1,1)}(t) &= \dt h^{(1,1)}(t). \label{evolution psi_111}
\end{align}
For $\dt \psi_1^{(2,0) + (0,2)}(t)$, it follows from (\ref{higher-order K terms}) that
\begin{align}
	 \dt \psi_1^{(2,0) + (0,2)}(t) &= \dt h^{(2,0) + (0,2)}(t) + \left(\dt \rho \right) * \rho, \label{evolution psi_12002}
\end{align}
since the $(2,0) + (0,2)$ components of $\psi_1(t)$ and $h(t)$ differ by terms that are quadratic in $\rho(t) = (h(t),K(t))$. Notice that by (\ref{linearization 3}) and (\ref{A bounds 3}) we have $\dt \rho$ is second order in $\psi$ and hence the final term in (\ref{evolution psi_12002}) may be absorbed in the error estimate $A$. Therefore from (\ref{evolution psi_2}), (\ref{evolution psi_111}), (\ref{evolution psi_12002}) and (\ref{linearization 3}) it follows that
\begin{align*}
	\dt \psi(t) = \LL(\psi(t))+A((\til{\omega},\til{J}),\psi(t)), 
\end{align*}
where $A$ is a different tensor than in (\ref{linearization 3}), but we still have $|A|_{C^k}\leq C(|\psi|_{C^k}|\nabla^2\psi|_{C^k}+|\nabla \psi|_{C^k}^2)$.
\qed

\medskip
Roughly speaking, the following theorem says that given any finite time $T>0$, by starting the flow very close to $(\til{\omega},\til{J})$, the solution $(\omega(t),J(t))$ remains close to $(\til{\omega},\til{J})$ on the interval $[0,T)$.   
\begin{thm} \label{start close, stay close} Given $T>0$, $\epsilon' > 0$ and an integer $k \geq 0$, there exists $\epsilon = \epsilon(T,\epsilon',k) > 0$ such that if $|\rho(0)|_{C^{\infty}}<\epsilon$ then, $(\omega(t),J(t))$ exists on $[0,T)$ and moreover $|\rho(t)|_{C^k}<\epsilon'$ on $[0,T)$.
\end{thm}

\begin{proof} First, for $\epsilon$ sufficiently small, work of Streets and Tian (see Theorem 1.1 in \cite{SCF}) shows that there exists $T' > 0$ such that the solution $(\omega(t),J(t))$ exists on $[0,T')$ and moreover $|\rho(t)|_{C^k} < \epsilon'$ on $[0,T')$. Suppose by way of contradiction that there exists a maximal $T''$ so that for all $\epsilon>0$ the solution exists and $|\rho(t)|_{C^k}<\epsilon'$ on $[0,T'')$ with $T''<T$. Fix $\til{T}<T''$. To derive a contradiction we will produce bounds on the $C^k$ norm of $\rho(t)$ on $[0,\til{T}]$ in terms of $\epsilon$, independent of $\til{T}$.\\
\indent Recall from Lemma \ref{psi is quadratic/cubic} that associated to $\rho(t)$ we have $\psi(t) \in T_{(\til{\omega},\til{J})}\mathcal{C}$. In order to obtain $C^k$ estimates on $\rho(t)$ in terms of $\epsilon$ we will produce $C^k$ bounds on $\psi(t)$ and employ (\ref{rho leq psi}). To this end, we study the evolution of $\psi(t)$. Recall from Lemma \ref{the flow of psi} part $(1)$ that
\begin{align} 
	\dt \psi(t)=\LL(\psi(t))+A((\til{\omega},\til{J}),\psi(t)) \label{linearization}
 \end{align}
 where $\LL$ is negative semi-definite and $A$ represents the error in approximating $(\FF, \GG)$ by $\LL$. From part $(2)$ of Lemma \ref{the flow of psi} we have
\begin{align}
	|A|_{C^k}\leq C(|\psi|_{C^k}|\nabla^2\psi|_{C^k}+|\nabla \psi|_{C^k}^2).\label{A bounds}
\end{align}
Notice that C depends on the $C^k$ norm of $\rho(t)$ which we are assuming is bounded by $\epsilon'$ for $t \in [0,T'') \supset [0,\til{T}]$.\\
\indent In the estimates that follow $\Rm$ will denote the curvature of the fixed metric $\til{g}$ and $\nabla$ will denote the Levi-Civita connection of $\til{g}$. Moreover we will use the fact that $M$ is compact and hence there exists a constant $C$ such that $|\Rm|_{C^{\infty}}<C$.

\subsection{$L^2$ bounds of $\psi$ in terms of $\epsilon$}
The linear stability of K\"{a}hler-Einstein structures will allow us to produce $L^2$ bounds on $\psi(t)$ in terms of $\epsilon$ which are independent of $\til{T}$. Indeed, for $t \in [0,\til{T}]$ by (\ref{linearization}) and using that $\langle \LL_{(\til{\omega},\til{J})} \cdot, \cdot \rangle_{L^2(\til{g})} \leq 0$, we have
\begin{align}
	\frac{1}{2}\dt \int_M|\psi|_{\til{g}}^2dvol_{\til{g}} = \int_M\bigg\langle \dt \psi, \psi \bigg\rangle \leq \int_MA*\psi.\label{2.14}
\end{align}
\indent Now, using the bound on $A$ given in (\ref{A bounds}), we see that 
\begin{align*}
	\int A*\psi = \int \psi^{*2}*\nabla^2 \psi + \nabla \psi^{*2}*\psi. 
\end{align*}
Using integration by parts on the second term yields 
\begin{align}
	\int_M A*\psi \leq  \int_M \psi^{*2}*\nabla^2 \psi. \label{2.15}  
\end{align}
For $t \in [0,\til{T}]$, by assumption and (\ref{psi leq rho}), $|\psi(t)|_{C^k} < \epsilon'$ therefore $\int_M \psi^{*2}*\nabla^2 \psi \leq C\epsilon' \int_M |\psi|^2$. Hence combining (\ref{2.14}) and (\ref{2.15}) we have
\begin{align}
	\frac{1}{2}\dt \int_M|\psi|_{\til{g}}^2dvol_{\til{g}}  \leq C_1\epsilon' \int_M|\psi|_{\til{g}}^2dvol_{\til{g}}.\label{L2 A bounds}
\end{align} 
Therefore for any $t \in [0,\til{T}]$,
\begin{align}
	|\psi(t)|_{L^2}^2 \leq e^{C_1\epsilon' T}|\psi_0|_{L^2}^2 \leq \epsilon e^{C_1\epsilon' T}.\label{L2}
\end{align}

\subsection{$L^{1,2}$ bounds of $\psi$ in terms of $\epsilon$}
Given the $L^2$ bounds above, we bootstrap to obtain higher-order bounds. Notice that linear stability was only used to start the bootstrapping process. Using (\ref{F11}), (\ref{F20}), (\ref{G}) and (\ref{linearization}) we have
\begin{align}
	\frac{1}{2}\dt \int_M|\psi|_{\til{g}}^2dvol_{\til{g}}&=\int_M\bigg\langle \dt \psi, \psi \bigg\rangle =\int_M\big\langle -\nabla^*\nabla \psi + \Rm*\psi + A, \psi \big\rangle.\label{L12 estimate}
\end{align}
\indent Since $M$ is compact there exists a constant $C_2$ such that $|\Rm(\til{g})|_{C^{\infty}} < C_2$. Moreover, we can bound the term associated with $A$ as we did in (\ref{2.15}) and (\ref{L2 A bounds}) to get
\begin{align}
	\frac{1}{2}\dt \int_M|\psi|_{\til{g}}^2dvol_{\til{g}} \leq -\int|\nabla \psi|^2dvol_{\til{g}}+C_3\int|\psi|^2dvol_{\til{g}}.\label{L2 evolution and estimate}
\end{align}

\noindent Integrating from $0$ to $\til{T}$, we see that
\begin{align*}
	\int_0^{\til{T}}\int_M|\nabla \psi|^2+\frac{1}{2}\int_M|\psi(\til{T})|^2 \leq \frac{1}{2}\int_M |\psi_0|^2+C_3\int_0^{\til{T}}		\int_M|\psi|^2. 
\end{align*}
Now the $L^2$ bounds from (\ref{L2}) imply
\begin{align}
	\int_0^{\til{T}}\int_M|\nabla \psi|^2 \leq C_4Te^{C_1\epsilon' T}|\psi_0|_{L^2}^2 \leq C_4Te^{C_1\epsilon' T}\epsilon.\label{L12}
\end{align}

\subsection{$L^{2,2}$ bounds of $\psi$ in terms of $\epsilon$}
Next, we use the $L^{1,2}$ bounds above to produce $L^{2,2}$ bounds. Similar to (\ref{L12 estimate}),
\begin{align}
	\frac{1}{2}\dt \int_M|\nabla \psi|_{\til{g}}^2dvol_{\til{g}}&=\int_M\big\langle \nabla(  -\nabla^*\nabla \psi + \Rm*\psi + A),\nabla \psi \big\rangle. \label{L12 evolution}
\end{align}
\indent First consider the term $\int \langle \nabla(  -\nabla^*\nabla \psi),\nabla \psi \rangle$ above. Commuting covariant derivatives and using integration by parts we get
\begin{align}
	\int_M \big\langle \nabla(  -\nabla^*\nabla \psi),\nabla \psi \big\rangle = -\int_M |\nabla^2\psi|^2 + \int_M \Rm * \nabla \psi * \nabla \psi. \label{first term L12 evolution}
\end{align}
\indent Next we obtain estimates on the term $\int \langle \nabla(\Rm * \psi), \nabla \psi \rangle = \int \langle \nabla \Rm * \psi + \Rm * \nabla \psi, \nabla \psi \rangle$ from equation (\ref{L12 evolution}). Since $|\Rm|_{C^{\infty}}<C_2$, we can use Young's Inequality, to show
 \begin{align}
 	\int_M \big\langle \nabla(\Rm * \psi), \nabla \psi \big\rangle \leq C'\int_M|\nabla \psi|^2 + C''\int_M|\psi|^2. \label{second term L12 evolution}
\end{align} 
\indent Finally, we consider the final term in (\ref{L12 evolution}), $\int \nabla A * \nabla \psi$. Using the estimates on $A$ from (\ref{A bounds}), we have 
\begin{align*}
	\int  \nabla A * \nabla \psi  = \int \nabla \psi^{*2} * \nabla^2\psi + \int \psi * \nabla \psi * \nabla^3\psi.
\end{align*} 
Integration by parts on the last term yields $\int \nabla A * \nabla \psi  = \int \nabla \psi^{*2} * \nabla^2\psi + \int \psi * \nabla^2\psi^{*2}$ and hence
\begin{align}
	\int_M \big\langle \nabla A, \nabla \psi \big\rangle \leq C''' \int_M |\nabla \psi|^2 + C_7\epsilon' \int_M |\nabla^2\psi|^2 \label{third term L12 evolution}
\end{align}
since $|\psi|_{C^k} < \epsilon'$ for $t \in [0,\til{T}]$.\\ 
\indent Combining (\ref{L12 evolution}), (\ref{first term L12 evolution}), (\ref{second term L12 evolution}), and (\ref{third term L12 evolution}) we see that
\begin{align}
	\frac{1}{2}\dt \int_M|\nabla \psi|_{\til{g}}^2dvol_{\til{g}}&\leq -\int_M| \nabla ^2 \psi|^2 + C_5\int_M|\psi|^2+ C_6 \int_M|\nabla \psi|^2 + C_7\epsilon' \int_M|\nabla^2\psi|^2.\label{1}
\end{align}
Hence, we choose $\epsilon'$ small enough so that $C_7\epsilon' < \frac{1}{2}$. Integrating (\ref{1}) from $0$ to $\til{T}$ we have
\begin{align}
	\int_0^{\til{T}}\int_M|\nabla^2 \psi|^2+\int_M|\nabla \psi(\til{T})|^2 &\leq \int_M |\nabla \psi_0|^2 + 2C_5\int_0^{\til{T}}\int_M|\psi|^2 + 2C_6\int_0^{\til{T}}\int_M|\nabla \psi|^2. \label{2.25}
\end{align}
Therefore, using the $L^2$ estimate from (\ref{L2}) and the $L^{1,2}$ estimate from (\ref{L12}) we have
\begin{align}
	\int_0^{\til{T}}\int_M|\nabla^2 \psi|^2 \leq C_7Te^{C_1\epsilon' T}|\psi_0|_{L^{1,2}}^2 \leq C_7Te^{C_1\epsilon' T}\epsilon. \label{L22}
\end{align}
\indent Notice that $(\ref{2.25})$ also gives bounds on $|\nabla \psi(\til{T})|_{L^2}^2$. Moreover by integrating (\ref{1}) from $0$ to $t$ for $t\in[0,\til{T}]$ these bounds hold not just at $\til{T}$ but for any $t\in[0,\til{T}]$. Hence we also have
$$\sup_{[0,\til{T}]}|\nabla \psi|_{L^2}^2 \leq C_7Te^{C_1\epsilon' T}\epsilon.$$

Now since $\dt \psi$ is second order in $\psi$, estimate (\ref{L22}) also gives $\int_0^{\til{T}}\int_M|\dt \psi|^2 \leq CTe^{C_1\epsilon' T}\epsilon$. Next we use induction to show that for any $p$ we have both:
$$\int_0^{\til{T}}\int_M|\nabla^p \psi|^2 \leq C(p)Te^{C_1\epsilon' T}\epsilon$$
and
$$\sup_{[0,\til{T}]}|\nabla^{p-1} \psi|_{L^2}^2 \leq C(p)Te^{C_1\epsilon' T}\epsilon.$$

\subsection{$L^{m+1,2}$ bounds on $\psi$ given $L^{m,2}$ bounds}
To produce $L^{m+1,2}$ bounds on $\psi$ given $L^{s,2}$ estimates for $s = 1, 2, \dots , m$ we compute the evolution of the $L^2$ norm of $\nabla^m\psi$.
\begin{align}
	\frac{1}{2}\dt \int_M|\nabla^m \psi|_{\til{g}}^2dvol_{\til{g}}&=\int_M\big\langle \nabla^m(  -\nabla^*\nabla \psi + \Rm*\psi + A),\nabla^m \psi \big\rangle. \label{Lr2 evolution}
\end{align}
\indent First we consider the term $\int \langle \nabla^m(-\nabla^*\nabla \psi), \nabla^m\psi \rangle$ above. Similar to, (\ref{first term L12 evolution}) commuting the covariant derivatives and using integration by parts yields
\begin{align*}
	\int_M \big\langle \nabla^m(-\nabla^*\nabla \psi), \nabla^m\psi \big \rangle = -\int_M |\nabla^{m + 1} \psi|^2 + \int_M \sum_{j = 0}^{m - 1} \nabla^j \Rm * \nabla^{m - j}\psi * \nabla^m \psi.
\end{align*}
Furthermore, using that $|\Rm|_{C^{\infty}}<C_2$ and employing Young's Inequality on each of the final $m - 1$ terms on the right-hand side we have that there exists a constant $C'$ such that:
\begin{align}
	\int_M \big\langle \nabla^m(-\nabla^*\nabla \psi), \nabla^m\psi \big \rangle \leq -\int_M |\nabla^{m + 1} \psi|^2 + C' \sum_{j = 0}^{m}|\nabla^j\psi|_{L^2}^2. \label{first term Lr2 evolution}
\end{align}
\indent Next we study the term $\int \langle \nabla^m(\Rm * \psi), \nabla^m\psi \rangle$ from equation (\ref{Lr2 evolution}). Again using that $|\Rm|_{C^{\infty}}<C_2$, by Young's Inequality we see that
\begin{align*} 
	\int_M \big\langle \nabla^m(\Rm * \psi), \nabla^m\psi \big\rangle = \int_M \nabla^m \psi * \nabla^m \psi + \cdots + \int_M \psi * \psi 
\end{align*}
and hence there exists a constant $C''$ such that
\begin{align} 
	\int_M \big\langle \nabla^m(\Rm * \psi), \nabla^m\psi \big\rangle \leq C'' \sum_{j = 0}^{m}|\nabla^j\psi|_{L^2}^2. \label{second term Lr2 evolution}
\end{align}
 
\indent Finally consider the term $\int \langle \nabla^m A, \nabla^m \psi \rangle$ from (\ref{Lr2 evolution}). Again we use the estimates on $A$ from (\ref{A bounds}). Here we have
\begin{align*} 
	\int_M \big \langle \nabla^m A, \nabla^m \psi \big \rangle \leq \int_M  \sum_{j = 0}^{m}\nabla^{j + 2}\psi * \nabla^{m - j}\psi * \nabla^m \psi + \int_M \sum_{j = 0}^m \nabla^{j + 1}\psi * \nabla^{m + 1 - j} \psi * \nabla^m \psi.
\end{align*}
\indent We will now show how to estimate the highest order terms in the right-hand side of the above inequality. First we rewrite the right-hand side as $\int \nabla^{m + 2}\psi * \nabla^m \psi * \psi + \nabla^{m + 1} \psi * \nabla^m \psi * \nabla \psi + \text{lower order terms}$. Integration by parts on the first term yields
\begin{align} 
	\int \langle \nabla^m A, \nabla^m \psi \rangle \leq \int \nabla^{m + 1}\psi * \nabla^{m + 1} \psi * \psi + \nabla^{m + 1} \psi * \nabla^m \psi * \nabla \psi + \text{lower order terms}. \label{Lr+1A1}
\end{align}
Next we use Young's Inequality on the second term on the right-hand side. In particular, Young's Inequality is written $ab \leq \eta a^2 + C(\eta) b^2$ where $\eta > 0$ can be taken arbitrarily small at the expense of making $C(\eta)$ large. Hence by Young's Inequality, 
\begin{align}
	\int \nabla^{m + 1} \psi * \nabla^m \psi * \nabla \psi \leq \eta \int |\nabla^{m + 1} \psi|^2 + C(\eta)\int |\nabla^m\psi|^2|\nabla \psi|^2. \label{Lr+1A2}
\end{align}
Therefore combining (\ref{Lr+1A1}) and (\ref{Lr+1A2}) and using that $|\psi|_{C^k} < \epsilon'$ for $t \in [0,\til{T}]$ we get
$\int_M \langle \nabla^m A, \nabla^m \psi \rangle \leq (C_8\epsilon' + \eta)\int_M |\nabla^{m + 1} \psi|^2 + C_9 \int_M |\nabla^m \psi|^2 + \text{lower order terms}$. Hence we choose $\eta = \frac{1}{4}$ and $\epsilon'$ sufficiently small so that 
\begin{align}
	C_8\epsilon' < \frac{1}{4}. \label{epsilon' C8}
\end{align} 
And so,
\begin{align}
	\int_M \langle \nabla^m A, \nabla^m \psi \rangle \leq \frac{1}{2}\int_M |\nabla^{m + 1} \psi|^2 + C_9 \int_M |\nabla^m \psi|^2 + \text{lower order terms}. \label{third term Lr2 evolution}
\end{align}

\indent We now have estimates for each term in the evolution of the $L^2$ norm of $\nabla^m \psi$ given in (\ref{Lr2 evolution}). In particular combining (\ref{Lr2 evolution}) (\ref{first term Lr2 evolution}), (\ref{second term Lr2 evolution}), and (\ref{third term Lr2 evolution}) we have
\begin{align}
	\frac{1}{2}\dt \int_M|\nabla^m \psi|_{\til{g}}^2dvol_{\til{g}} \leq -\frac{1}{2}\int_M |\nabla^{m + 1} \psi|^2 + C_{10}\sum_{j = 0}^{m}|\nabla^j\psi|_{L^2}^2. \label{Lk2 evolution}
\end{align}
Integrating from $0$ to $\til{T}$ we get
\begin{align}
	\int_0^{\til{T}}\int_M|\nabla^{m + 1} \psi|^2+\int_M|\nabla^m \psi(\til{T})|^2 \leq \int_M |\nabla^m \psi(0)|^2 + 2C_{10}\int_0^{\til{T}}\sum_{j = 0}^{m}|\nabla^j\psi|_{L^2}^2. \label{k+1}
\end{align}

Now we can employ the $L^{s,2}$ estimates for $s = 1, \dots, m$ to get $L^{m+1,2}$ bounds. In particular,
\begin{align*}
	\int_0^{\til{T}}\int_M|\nabla^{m + 1} \psi|^2 \leq C_{11}Te^{C_1\epsilon' T}|\psi_0|_{L^{m,2}}^2 \leq C_{11}Te^{C_1\epsilon' T}\epsilon.
\end{align*}
\indent Notice that $(\ref{k+1})$ also gives bounds on $|\nabla^m \psi(\til{T})|_{L^2}^2$. Moreover by integrating (\ref{Lk2 evolution}) from $0$ to $t$ for any $t\in[0,\til{T}]$, the bound on $|\nabla^m \psi|_{L^2}^2$ holds not just at $\til{T}$ but for any $t\in[0,\til{T}]$. Hence we also have
$$\sup_{[0,\til{T}]}|\nabla^m \psi|_{L^2}^2 \leq C_{11}Te^{C_1\epsilon' T}|\psi_0|_{L^{m,2}}^2 \leq C_{11}Te^{C_1\epsilon' T}\epsilon.$$

This proves that for any $p$
\begin{align}
\int_0^{\til{T}}\int_M|\nabla^p \psi|^2 \leq C(p)Te^{C_1\epsilon' T}|\psi_0|_{L^{p-1,2}}^2 \leq C(p)Te^{C_1\epsilon' T}\epsilon \label{Lp bound 1}
\end{align}
and
\begin{align}
	\sup_{[0,\til{T}]}|\nabla^{p-1} \psi|_{L^2}^2 \leq C(p)Te^{C_1\epsilon' T}|\psi_0|_{L^{p-1,2}}^2 \leq C(p)Te^{C_1\epsilon' T}\epsilon.\label{Lp bound 2}
\end{align}
Furthermore, since $\dt \psi$ is second order in $\psi$, (\ref{Lp bound 1}) also implies that $\int_0^{\til{T}}\int_M\big|\frac{\del^q}{\del t^q}\nabla^r \psi\big|^2\leq C\epsilon$ for any $q,r>0$, where C is independent of $\til{T}$.\\ 
\indent Now use the Sobolev Embedding Theorem, with respect to $\til{g}$, to obtain $C^k$ bounds on $\psi$ in terms of $\epsilon$. And hence by (\ref{rho leq psi}) we have $C^k$ bounds on $\rho$ in terms of $\epsilon$. In \cite{SCF}, Theorem 1.9, Streets and Tian prove that if there is a finite time singularity $\tau$ of the flow, then $\lim_{t \rightarrow \tau} \sup \{ |\Rm|_{C^0}, |DT|_{C^0}, |T|^2_{C^0}   \} = \infty$. Here $D$ denotes the Levi-Citia connection and $\Rm$ is the curvature of $D$. Therefore, the fact that the estimates above are independent of $\til{T}$ implies that the solution exists on $[0,T'']$. Again, using the short-time existence result of Streets and Tian (\cite{SCF}), the solution can be extended past time $T''$. Moreover, for $\epsilon$ sufficiently small, we maintain the $C^k$ estimates on $\rho(t)$ past time $T''$. This contradicts the maximality of $T''$.
\end{proof}

\medskip

\section{Dynamic Stability when $c_1(\til{J}) < 0$}\label{c1<0}
In this section we prove that when $c_1(\til{J}) < 0$, VNAHCF converges exponentially to the K\"{a}hler-Einstein structure $(\til{\omega},\til{J})$. As above we let $\rho(t) = (\omega(t)-\til{\omega},J(t) - \til{J})$.

\begin{thm} \label{dynamic stability c1 < 0} 
Let $(M^{2n},\til{\omega},\til{J})$ be a closed complex manifold where $(\til{\omega},\til{J})$ is a K\"{a}hler-Einstein structure such that $c_1(\til{J}) < 0$. Given a positive integer $k$, there exists $\epsilon = \epsilon(k) >0$ such that if $(\omega(0),J(0))$ is an almost hermitian structure with $\big|\rho(0)\big|_{C^{\infty}} < \epsilon$, then the solution to the volume-normalized AHCF starting at $(\omega(0),J(0))$ exists for all time and converges exponentially in $C^k$ to $(\til{\omega},\til{J})$.
\end{thm}

\proof We prove Theorem \ref{dynamic stability c1 < 0} using two lemmas. As in Section 2 we have to deal with the non-linearity of the space of almost hermitian structures. To prove Theorem \ref{dynamic stability c1 < 0} we will show that there exists $\epsilon$ so that if $\big|\rho(0)\big|_{C^{\infty}} < \epsilon$, then $\psi(t)$ exponentially decays in $C^k$. Finally employing (\ref{rho leq psi}) exponential $C^k$ decay of $\rho(t)$ will follow from exponential $C^k$ decay of $\psi(t)$.

\begin{lemma}\label{c1 < 0 L2 decay} Given $\delta > 0$ and an integer $k \geq 0$, there exists $\epsilon_1 = \epsilon_1(\delta, k) > 0$ such that if $\big|\rho(0)\big|_{C^{\infty}} < \epsilon_1$ then $|\psi(t)|_{C^k} < \delta$ for all $t \geq 0$ and moreover $|\psi(t)|_{L^2}^2 \leq Ce^{-\lambda t}$ for all $t \geq 0$.
\end{lemma}
  
\proof As in Section \ref{linstab}, let $\lambda = \min \{ |\lambda_i| : \text{$\lambda_i$ is an eigenvalue of $\LL$} \}$. Further  let $\psi(t) \in T_{(\til{\omega},\til{J})}\mathcal{C}$ be the element of the tangent space, from Lemma \ref{psi is quadratic/cubic}, associated to $\rho(t)$. Recall that from Corollary \ref{c1<0 implies LL<0} that $c_1(\til{J}) < 0$ implies that
\begin{align}
	\int_M \big\langle \LL_{(\til{\omega},\til{J})} \psi, \psi \big\rangle_{\til{g}} \leq -\lambda|\psi|_{L^2(\til{g})}^2. \label{LLpsi leq psi}
\end{align} 
And by (\ref{linearization}),
\begin{align}
\frac{1}{2}\dt \int_M|\psi|_{\til{g}}^2dvol_{\til{g}} = \int_M\bigg\langle \dt \psi, \psi \bigg\rangle = \int_M\bigg\langle \LL \psi + A, \psi \bigg\rangle. \label{c1 < 0 L2 evolution}
\end{align}

Then for any $t$ for which $|\psi(t)|_{C^2} < {\delta}$, we can employ the bound
\begin{align}
	\int \langle A, \psi \rangle \leq C_1{\delta}|\psi|_{L^2}^2 \label{A delta bound}
\end{align} 
derived in (\ref{2.15}) and (\ref{L2 A bounds}). Combining (\ref{LLpsi leq psi}), (\ref{c1 < 0 L2 evolution}) and (\ref{A delta bound}) yields

\begin{align*}
	\frac{1}{2}\dt \int_M|\psi|_{\til{g}}^2dvol_{\til{g}} \leq -\lambda|\psi|_{L^2}^2 + C_1{\delta}|\psi|_{L^2}^2.
\end{align*}
Here we choose ${\delta}$ so that
\begin{align}
	C_1{\delta} < \frac{1}{2}\lambda. \label{c1 < 0 epsilon'}
\end{align}
Integrating from $0$ to $t$ yields $L^2$ exponential decay of $\psi(t)$. In particular,
\begin{align}
	|\psi(t)|_{L^2}^2 \leq e^{-\lambda t}|\psi(0)|_{L^2}^2 \label{L2 decay c1 < 0} 
\end{align}
for any $t$ for which $|\psi(t)|_{C^2} < {\delta}$. Therefore to complete the lemma, we will show that given $k \geq 2$, there exists $\epsilon_1(k, \delta) > 0$ such that if $|\rho(0)|_{C^{\infty}} < \epsilon_1$, then $|\psi(t)|_{C^k} < {\delta}$ for all $t \in [0,\infty)$. Notice again that from (\ref{psi leq rho}) it follows that $|\rho(0)|_{C^{\infty}} < \epsilon_1$ implies $|\psi(0)|_{C^{\infty}} < \epsilon_1$.\\
\indent By Theorem \ref{start close, stay close} we know that given $T > 0$, $k \geq 0$ and $\epsilon'>0$, there exists $\epsilon > 0$ such that $|\rho(0)|_{C^{\infty}} < \epsilon$ implies that $|\psi(t)|_{C^k} < \epsilon'$ on $[0,T)$. We apply Theorem \ref{start close, stay close} with $\epsilon' = {\delta}$ and ${\delta}$ sufficiently small so that (\ref{c1 < 0 epsilon'}) holds. Let $\epsilon_2$ denote a constant that is small enough so that if $|\psi(0)|_{C^{\infty}} < \epsilon_2$ then $|\psi(t)|_{C^k} < {\delta}$ on $[0,T)$ and assume that 
\begin{align}
	|\psi(0)|_{C^{\infty}} < \epsilon_2. \label{epsilon2}
\end{align}
\indent Given $T$, let $t_0 < t<T$, then integrating (\ref{L2 evolution and estimate}) from $t_0$ to $t$ we have
\begin{align}
	\int_{t_0}^{t}|\nabla \psi|_{L^2}^2 \leq \frac{1}{2}|\psi(t_0)|_{L^2}^2+C_3\int_{t_0}^{t}|\psi(s)|_{L^2}^2. \label{c1 < 0 1}
\end{align}
Furthermore since (\ref{L2 decay c1 < 0}) holds on $[0,T)$,
\begin{align}
	\int_{t_0}^{t}|\psi(s)|_{L^2}^2 \leq \int_{t_0}^{t}e^{-\lambda s}|\psi(0)|_{L^2}^2 = \frac{1}{\lambda}e^{-\lambda t_0}|\psi(0)|_{L^2}^2. \label{c1 < 0 2}
\end{align}
Therefore combining (\ref{c1 < 0 1}) and (\ref{c1 < 0 2}) yields
\begin{align}
	\int_{t_0}^{t}|\nabla \psi|_{L^2}^2 \leq \frac{1}{2}|\psi(t_0)|_{L^2}^2+\frac{C_3}{\lambda}e^{-\lambda t_0}|\psi(0)|_{L^2}^2 \leq \left(\frac{1}{2} + \frac{C_3}{\lambda}\right)e^{-\lambda t_0}|\psi(0)|_{L^2}^2. \label{c1 < 0 3}
\end{align}
The last inequality is again by (\ref{L2 decay c1 < 0}). The key observation here is that the $L^{1,2}$ estimate in (\ref{c1 < 0 3}) is independent of $t$.

Next, to obtain a similar $L^{2,2}$ estimate we integrate (\ref{1}) from $t_0$ to $t$.
\begin{align}
	 \int_{t_0}^{t}|\nabla^2 \psi|_{L^2}^2 + |\nabla \psi(t)|_{L^2}^2 \leq |\nabla \psi(t_0)|_{L^2}^2+C_5\int_{t_0}^{t}|\psi(t)|_{L^2}^2+C_6\int_{t_0}^{t}|\nabla \psi(t)|_{L^2}^2.\label{c1 < 0 4}
\end{align}
Bounding the last two terms of (\ref{c1 < 0 4}) using (\ref{c1 < 0 2}) and (\ref{c1 < 0 3}) yields
\begin{align*}
	 \int_{t_0}^{t}|\nabla^2 \psi|_{L^2}^2 + |\nabla \psi(t)|_{L^2}^2 \leq |\nabla \psi(t_0)|_{L^2}^2+Ce^{-\lambda t_0}|\psi(0)|_{L^2}^2.
\end{align*}
Again the key observation is that the estimate above is independent of $t$.\\ 
\indent Using the same inductive argument as in the proof of Theorem \ref{start close, stay close} shows that for any $p$,
\begin{align}
	 \int_{t_0}^{t}|\nabla^p \psi|_{L^2}^2 + |\nabla^{p-1} \psi(t)|_{L^2}^2 \leq C_1(p)|\psi (t_0)|_{L^{p-1,2}}^2+C_2(p)e^{-\lambda t_0}\label{c1 < 0 6}
\end{align} 
where $C_1(p)$ and $C_2(p)$ are independent of $t$. Notice that there exists a constant $C$, such that $|\psi (t_0)|_{L^{p-1,2}}^2 \leq C|\psi (t_0)|_{C^{p-1}}$. Hence by (\ref{c1 < 0 6}) we have $|\nabla^{p-1} \psi(t)|_{L^2}^2 \leq C_1'(p)|\psi (t_0)|_{C^{p-1}}+C_2(p)e^{-\lambda t_0}$. Therefore applying the Sobolev Embedding Theorem we have
\begin{align}
	|\psi(t)|_{C^k} \leq C_1(k)|\psi(t_0)|_{C^{p-1}} + C_2(k)e^{-\lambda t_0} \label{c1 < 0 7}
\end{align}
where $C_1(k)$ and $C_2(k)$ are independent of $t$.\\ 
\indent Since (\ref{c1 < 0 7}) is independent of t, to prove that $|\psi(t)|_{C^k} < {\delta}$ for all $t\in [0,\infty)$, it suffices to show that there exists a constant $\epsilon_1$ with $0 < \epsilon_1 \leq \epsilon_2$ and such that $|\rho(0)|_{C^{\infty}} < \epsilon_1$ implies that the right-hand side of (\ref{c1 < 0 7}) is bounded above by ${\delta}$.\\
\indent First we bound the second term on the right-hand side of (\ref{c1 < 0 7}). Notice that, given ${\delta} > 0$ small enough so that we have (\ref{c1 < 0 epsilon'}), the argument above which led to inequality (\ref{c1 < 0 7}) holds under the assumption (\ref{epsilon2}). Furthermore, notice that the estimate in (\ref{c1 < 0 7}) holds for $t_0 < T$, independent of $T$. Therefore we take $T$ to be sufficiently large so that $T>t_0$ and 
\begin{align}
	C_2(k)e^{-\lambda t_0} < \frac{1}{2}{\delta}. \label{epsilon2 requirement}
\end{align}
\indent To bound the first term in (\ref{c1 < 0 7}) we again use Theorem \ref{start close, stay close} with $\epsilon' = \frac{1}{2C_1(k)} {\delta}$ and $T > t_0$. Hence, by Theorem \ref{start close, stay close} there exists $\epsilon_3 > 0$ such that $|\rho(0)|_{C^{\infty}} < \epsilon_3$ implies that
\begin{align}
	|\psi|_{C^{p-1}} < \epsilon' = \frac{1}{2C_1(k)} {\delta} \label{epsilon3}
\end{align}
for $t \in [0,T) \supset [0,t_0]$.\\
\indent Finally, choose $\epsilon_1 = min \{ \epsilon_2, \epsilon_3 \}$. Hence combining (\ref{c1 < 0 7}), (\ref{epsilon2 requirement}) and (\ref{epsilon3}) proves that if $|\psi(0)|_{C^{\infty}} < \epsilon_1$, then (\ref{c1 < 0 7}) holds independent of $t$. Therefore it follows that $|\psi(t)|_{C^k} < {\delta}$ for all $t \geq 0$ and moreover the $L^2$ decay estimate in (\ref{L2 decay c1 < 0}) holds for all $t \geq 0$.
\qed

\medskip
\indent To finish the proof of Theorem \ref{dynamic stability c1 < 0}, we show that the $L^2$ decay estimate above and parabolic theory can be used to prove $C^k$ decay of $\psi(t)$.

\begin{lemma}\label{L2 to Ck decay} Given an integer $k \geq 2$, there exists ${\delta} = {\delta}(k) > 0$ such that if both $|\psi(t)|_{C^k} < {\delta}$ for all $t\in[0,\infty)$ and $|\psi(t)|_{L^2}^2 \leq Ce^{-\lambda t}$ then $|\psi(t)|_{C^k} \leq C(k)e^{-\lambda t}$.
\end{lemma}
\proof We begin the proof by deriving an $L^{1,2}$ exponential decay estimate. The same argument that was used to derive (\ref{L12 estimate}) shows that there exists a constant $C_1$ such that
\begin{align*}
	\dt |\psi|_{L^2}^2 \leq -|\nabla \psi|_{L^2}^2 + C_1|\psi|_{L^2}^2, 
\end{align*}
where $C_1$ depends on both $(\til{\omega},\til{J})$ and $|\psi(t)|_{C^2}$; but by assumption $|\psi|_{C^2} < \delta$ for all $t \geq 0$. Integrating from $t$ to $\infty$ yields,
\begin{align}
	\int_{t}^{\infty}|\nabla \psi|_{L^2}^2 \leq |\psi(t)|_{L^2}^2 + C_1 \int_{t}^{\infty}|\psi|_{L^2}^2 \leq Ce^{-\lambda t}. \label{c1 < 0 9}
\end{align}
The last inequality follows from the assumed $L^2$ exponential decay estimate.\\
\indent Next, for a fixed $t$, let $\theta(s)$ be a smooth function which is $0$ for $s \in \left[t-1,t-\frac{1}{2}\right]$, monotonically increasing from $0$ to $1$ for $s \in \left[t - \frac{1}{2}, t \right]$ and $1$ for $s \geq t$. As we shall see below, $\theta(s)$ will be used to deal with boundary terms which arise in the parabolic estimates that follow. The same argument that was used to produce (\ref{1}) shows that there exist constants such that
\begin{align}
	\frac{1}{2}\dt \int_M|\nabla \psi|_{\til{g}}^2dvol_{\til{g}}&\leq -\int_M| \nabla ^2 \psi|^2 + C_2\int_M|\psi|^2+ C_3 \int_M|\nabla \psi|^2 + C_4\delta \int_M|\nabla^2\psi|^2, \label{L2 to Ck decay 1}
\end{align}
again these constants depend on both $(\til{\omega},\til{J})$ and $|\psi(t)|_{C^2}$. Now we choose $\delta$ sufficiently small so that $C_4\delta < \frac{1}{2}$. Hence using (\ref{L2 to Ck decay 1}) and that both $\theta(s)$ and its derivative are uniformly bounded,
\begin{align}
	\frac{\partial}{\partial s} \left( \theta(s)|\nabla \psi (s)|_{L^2}^2 \right) \leq C_5|\nabla \psi(s)|_{L^2}^2 + C_6|\psi(s)|_{L^2}^2. \label{c1 < 0 10}
\end{align}
We integrate (\ref{c1 < 0 10}) in $s$ from $t - \frac{1}{2}$ to $t$ for $t \geq 1$. Using that $\theta \left(t - \frac{1}{2}\right) = 0$ and $\theta(t) = 1$ we get
\begin{gather}
\begin{split} \label{by above}
	|\nabla \psi(t)|_{L^2}^2 &\leq C_5\int_{t - \frac{1}{2}}^{t}|\nabla \psi(s)|_{L^2}^2 + C_6\int_{t - \frac{1}{2}}^{t}|\psi(s)|_{L^2}^2\\
	&\leq C_5\int_{t - \frac{1}{2}}^{\infty}|\nabla \psi(s)|_{L^2}^2 + C_6\int_{t - \frac{1}{2}}^{\infty}|\psi(s)|_{L^2}^2. 
\end{split}
\end{gather}
Hence using the $L^2$ decay assumption and (\ref{c1 < 0 9}) it follows from (\ref{by above}) that
\begin{align}
	|\nabla \psi(t)|_{L^2}^2 \leq Ce^{-\lambda t}. \label{c1 < 0 11}
\end{align}
This proves exponential $L^{1,2}$ decay.\\
\indent Next we prove $L^{2,2}$ decay. By (\ref{L2 to Ck decay 1}) with $\delta$ small enough so that $C_4 \delta < \frac{1}{2}$,
\begin{align*}
	\dt |\nabla \psi|_{L^2}^2 \leq -|\nabla^2 \psi|_{L^2}^2 + 2C_3|\nabla \psi|_{L^2}^2 + 2C_2|\psi|_{L^2}^2. 
\end{align*}
Integrating from $t$ to $\infty$ yields
\begin{align}
	\int_{t}^{\infty}|\nabla^2 \psi|_{L^2}^2 \leq |\nabla \psi(t)|_{L^2}^2 + 2C_3\int_{t}^{\infty}|\nabla \psi|_{L^2}^2 + 2C_2\int_{t}^{\infty}|\psi|_{L^2}^2 \leq Ce^{-\lambda t}. \label{c1 < 0 13}
\end{align}
Where (\ref{c1 < 0 11}), (\ref{c1 < 0 9}), and the $L^2$ decay assumption were used in the first, second, and third term on the right-hand side respectively.\\
\indent Now, as in (\ref{epsilon' C8}) we choose $\delta$ small enough so that (\ref{Lk2 evolution}) holds for $m = 2$. Therefore there exists a constant $C_6$ such that
\begin{align*}
	\frac{1}{2}\dt \int_M|\nabla^2 \psi|_{\til{g}}^2dvol_{\til{g}} \leq -\frac{1}{2}\int_M |\nabla^{3} \psi|^2 + C_{6}\sum_{j = 0}^{2}|\nabla^j\psi|_{L^2}^2.
\end{align*}
Hence, there exists a constant $C_7$ so that
\begin{align}
	\frac{\partial}{\partial s} \left( \theta(s)|\nabla^2 \psi (s)|_{L^2}^2 \right) \leq C_7|\nabla^2 \psi(s)|_{L^2}^2 + C_7|\nabla \psi(s)|_{L^2}^2 + C_7| \psi(s)|_{L^2}^2. \label{c1 < 0 14}
\end{align}
We integrate (\ref{c1 < 0 14}) in $s$ from $t - \frac{1}{2}$ to $t$ for $t \geq 1$. Using that $\theta \left(t - \frac{1}{2}\right) = 0$ and $\theta(t) = 1$ we get
\begin{align*}
	|\nabla^2 \psi(t)|_{L^2}^2 \leq C_7\int_{t - \frac{1}{2}}^{t}\sum_{j = 0}^{2}|\nabla^j \psi(s)|_{L^2}^2 \leq C_7\int_{t - \frac{1}{2}}^{\infty}\sum_{j = 0}^{2}|\nabla^j \psi(s)|_{L^2}^2.
\end{align*}
Hence using (\ref{c1 < 0 13}), (\ref{c1 < 0 9}) and the assumed $L^2$ decay we get
\begin{align*}
	|\nabla^2 \psi(t)|_{L^2}^2 \leq Ce^{-\lambda t}. 
\end{align*}
This gives exponential $L^{2,2}$ decay.\\
\indent Continuing in this way we get
\begin{align*}
	|\psi(t)|_{L^{p,2}}^2 \leq C(p)e^{-\lambda t}.
\end{align*}
Furthermore by the Sobolev Embedding Theorem we get $|\psi(t)|_{C^k} \leq C(k)e^{-\lambda t}$. 
\qed

\medskip
\indent By Lemma \ref{L2 to Ck decay} we know that given $k \geq 2$, there exists $\delta > 0$ so that if both $|\psi(t)|_{C^k} < \delta$ and $|\psi(t)|_{L^2}^2 \leq Ce^{-\lambda t}$ hold for all $t \geq 0$, then $|\psi(t)|_{C^k} \leq C(k)e^{-\lambda t}$ for all $t \geq 0$. Furthermore by Lemma \ref{c1 < 0 L2 decay} we know that there exists $\epsilon_1 > 0$ such that if $|\psi(0)|_{C^{\infty}} < \epsilon_1$, then both $|\psi(t)|_{C^k} < \delta$ and $|\psi(t)|_{L^2}^2 \leq Ce^{-\lambda t}$ hold for all $t \geq 0$. Hence let $\delta$ be determined by Lemma \ref{L2 to Ck decay}. To finish the proof of Theorem \ref{dynamic stability c1 < 0}, we apply Lemma \ref{c1 < 0 L2 decay} with $\epsilon = \epsilon_1$ and note that by (\ref{rho leq psi}) exponential decay of $\rho(t)$ follows from exponential decay of $\psi(t)$.
\qed 

\medskip

\section{Dynamic Stability in the Calabi-Yau Case}
In Section \ref{c1<0} we proved Theorem \ref{dynamic stability} when $c_1(\til{J}) < 0$ by using that, in this case, the linearization $\LL$ is negative definite. However in the Calabi-Yau case, the kernel of $\LL$ is non-trivial and so the non-linear part of the flow is no longer controlled by the linear part. In this section we will show that in the Calabi-Yau case we can find a Calabi-Yau structure to which the flow exponentially converges.\\

\indent In order to find a Calabi-Yau structure to which the flow exponentially converges we will construct a sequence $\{ (\omega_j,J_j) \}$ of successively better Calabi-Yau structures; in the sense that the solution $(\omega(t),J(t))$ to the VNAHCF converges exponentially on larger and larger intervals. Moreover we will prove that each of these Calabi-Yau structures is contained in a fixed neighborhood of the original Calabi-Yau structure (see Theorem \ref{new KE structure} part $(2)$). This will allow us to extract a limit $(\omega_{KE},J_{KE})$ to which $\{ (\omega_j,J_j) \}$ subconverges. One could imagine that if this sequence failed to converge that we would only be able to conclude that the solution becomes asymptotic to the space of Calabi-Yau structures.\\
\indent In order to choose a new Calabi-Yau structure we will use Koiso's Theorem. Before stating Koiso's Theorem we need a definition.

\defn Let $\mathcal{AC}$ denote the space of almost complex structures on $M$ modulo diffeomorphism. A complex structure $J$ is \emph{unobstructed} if for any $\dot{J} \in T_J\mathcal{AC}$ such that $\dot{N}(\dot{J}) = 0$, there exists a path of complex structures $J(a)$ such that $J(0) = J$ and $\da\big|_{a=0}J(a)=\dot{J}$. Again, $N$ denotes the Nijenhuis tensor.

\begin{thm}\label{Koiso's Theorem} (Koiso \cite{KoisoKEstructures}) Let $(\omega, J)$ be a K\"{a}hler-Einstein structure on $M$. Assume that:
	\begin{enumerate}
		\item the first Chern class of $J$ is zero;
		\item $J$ is unobstructed.
	\end{enumerate}
Then the space of K\"{a}hler-Einstein structures, modulo diffeomorphism, around $(\omega, J)$ is a manifold.
\end{thm}

In order to make use of Koiso's Theorem we employ a theorem of Tian and Todorov.

\begin{thm}\label{Tian and Todorov}  (Tian \cite{ComplexDef1} and Todorov \cite{ComplexDef2}) Let $(M,J)$ be a closed Calabi-Yau manifold. Then $J$ is unobstructed.
\end{thm}

Next we will describe the tangent space of Calabi-Yau structures at $(\til{\omega},\til{J})$. Using the above two theorems we prove that the kernel of $\LL$ is isomorphic to the tangent space of Calabi-Yau structures at $(\til{\omega},\til{J})$. Let $\mathcal{U}$ denote the space of Calabi-Yau structures near $(\til{\omega},\til{J})$ modulo diffeomorphism. 

\lemma \label{tangent space lemma} Let $(M, \til{\omega}, \til{J})$ be a closed Calabi-Yau manifold, then $T_{(\til{\omega},\til{J})}\mathcal{U} \cong \Ker \LL$.

\proof Let $(\omega(a),J(a))$ be a one-parameter family of unit volume almost hermitian structures and write $\da \big|_{a = 0}(\omega(a),J(a)) = (\dot{\omega},\dot{J})$. First since Calabi-Yau structures are static under the system (\ref{flow of w,J}), we have $T_{(\til{\omega},\til{J})}\mathcal{U} \subseteq \Ker \LL$. To prove $\Ker \LL \subseteq T_{(\til{\omega},\til{J})}\mathcal{U}$, let $(\dot{\omega},\dot{J}) \in \Ker \LL$.\\
\indent We make the following claim. \emph{If $\dot{J} \in \ker \dot{\GG}$ then $\dot{J} \in \ker \dot{N}$.}
To see this, first notice that from (\ref{linearization of G}) and using that the scalar curvature $s_{\til{g}} = 0$, if $\dot{J} \in \ker \dot{\GG}$, then $\Delta_{\overline{\del}}\dot{J} = 0$. By integrating we see that $\overline{\del} \dot{J} = 0$. On the other hand, in coordinates, the Nijenhuis tensor is written

	$$N_{j k}^{i} = J_j^{p} \del_{p}J_k^{i} - J_k^p\del_pJ_j^i - J_p^i\del_jJ_k^p + J_p^i\del_kJ_j^p.$$
	
\noindent Hence,
\begin{align*}
	\dot{N}_{j k}^{i} &= \dot{J}_j^{p} \del_{p}J_k^{i} + J_j^{p} \del_{p}\dot{J}_k^{i} - \dot{J}_k^p\del_pJ_j^i - J_k^p\del_p\dot{J}_j^i - \dot{J}_p^i\del_jJ_k^p - J_p^i\del_j\dot{J}_k^p + \dot{J}_p^i\del_kJ_j^p + J_p^i\del_k\dot{J}_j^p.
\end{align*}
Now each of the terms above of the form $\dot{J}*\del J$ can be written as $\dot{J}*\del J = \dot{J}( \nabla J + \Gamma *J)$. So using normal, complex coordinates (with respect to the Calabi-Yau structure $(\til{g},\til{J})$), at a point $p \in M$, we have that each of these terms vanish. Here we also made use of the fact that when $(\til{g},\til{J})$ is K\"{a}hler the Chern connection coincides with the Levi-Civita connection and so $\til{J}$ is parallel with respect to the connection.
Next, since $\dot{J} \in \Lambda^{0,1} \otimes T^{1,0}$ in these normal, complex coordinates at $p \in M$, we have,
\begin{align*}
	\dot{N}_{\bj \bk}^{i} &= J_{\bj}^{\bp} \del_{\bp}\dot{J}_{\bk}^{i} - J_{\bk}^{\bp}\del_{\bp}\dot{J}_{\bj}^i - J_p^i\del_{\bj}						\dot{J}_{\bk}^p + J_p^i\del_{\bk}\dot{J}_{\bj}^p = 0.
\end{align*}
This proves the claim.

By Theorem \ref{Tian and Todorov} there exists a path of complex structures ${J}(a)$ where $J(0) = {J}$ and $\frac{d}{da} J \big|_{a=0} = \dot{J}$. Next, using (\ref{linearization of F}) we have that $\dot{\omega} \in \ker \dot{\FF}$ implies that $\Delta_d\dot{\omega} = 0$, that is $\dot{\omega}$ is harmonic. Therefore, by the Calabi-Yau Theorem (\cite{Aub1}, \cite{Yau1}, \cite{Yau2}, also see Theorem 2.29 in \cite{Chow1}) $\omega(a)$ is a variation through K\"{a}hler metrics such that $[\omega(a)] = [\omega_{KE}(a)]$, where $\omega_{KE}(a)$ is Ricci-flat. Moreover by the Hodge Decomposition Theorem there is a unique harmonic representative in each cohomology class. Hence $\omega(a) = \omega_{KE}(a)$ and we have that $\dot{\omega}$ arises as a variation through Calabi-Yau metrics.
\qed
 
\medskip
\indent Notice that $\Lambda^2(M) \times \End(TM)$ is an affine space which can be viewed as a vector space by taking $(\til{\omega},\til{J})$ to be the origin. Throughout this section we will view $\Lambda^2(M) \times \End(TM)$ as a vector space. Let $\pi_0 : \Lambda^2(M) \times \End(TM) \rightarrow \Ker \LL$ be the projection onto the kernel of $\LL$.\\  
\indent Let $(\til{\omega},\til{J})$ denote the Calabi-Yau structure from Theorem \ref{dynamic stability}. Roughly speaking, we will next prove that there exists a better Calabi-Yau structure $(\omega_I,J_I)$; in the sense that the solution $(\omega(t),J(t))$ to VNAHCF exponentially converges to $(\omega_I,J_I)$ on an interval $I$ (see Theorem \ref{new KE structure} and Lemma \ref{exponential decay}). Throughout this section $\rho_I(t) \doteq (\omega(t) - \omega_I,J(t) - J_I)$ will quantify the distance the solution is from this new Calabi-Yau structure. As above let $\rho(t) = (\omega(t) - \til{\omega},J(t) -  \til{J})$. Notice that we may view both $\rho(t)$ and $\rho_I(t)$ as elements of $\Lambda^2(M) \times \End(TM)$.\\
\indent As in Section \ref{linstab} we have to deal with the non-linearity of the space of almost hermitian structures modulo diffeomorphism denoted $\mathcal{C}$. Notice that we may write $\rho_I(t) = \rho(t) - \til{\rho}_I$ where $\til{\rho}_I \doteq (\omega_I - \til{\omega},J_I - \til{J})$. From Lemma \ref{psi is quadratic/cubic} associated to $\rho(t)$ we have $\psi(t) \in T_{(\til{\omega},\til{J})}\mathcal{C}$ and analogously associated to $\til{\rho}_I$ we have $\til{\psi}_I \in T_{(\til{\omega},\til{J})}\mathcal{C}$. Hence associated to $\rho_I(t)$ we have ${\psi}_I(t) \in T_{(\til{\omega},\til{J})}\mathcal{C}$ defined by
\begin{align*}
	\psi_I(t) \doteq \psi(t) - \til{\psi}_I.
\end{align*}
Moreover, employing the same argument as in the proof of Lemma \ref{psi is quadratic/cubic} we have
\begin{align}
	|\psi_I(t)|_{C^k} \leq |\rho_I(t)|_{C^k} \label{psi_I leq rho_I},
\end{align}
\begin{align}
	|\rho_I(t)|_{L^2} \leq |\psi_I(t)|_{L^2} + C|\psi_I(t)|_{L^2}^2 \label{rho_I leq psi_I L2}
\end{align}
and
\begin{align}
	|\rho_I(t)|_{C^k} \leq |\psi_I(t)|_{C^k} + C|\psi_I(t)|_{C^k}^2 \label{rho_I leq psi_I}.
\end{align}
Similarly by the proof of Lemma \ref{the flow of psi} we have
\begin{align}
	\dt \psi_I(t) &= \LL (\psi_I(t)) + A((\til{\omega},\til{J}),\psi_I(t)) \label{evolution of psi_I(t)}
\end{align}
where
\begin{align}
	|A((\til{\omega},\til{J}),\psi_I(t))|_{C^k} &\leq C(|\psi_I(t)|_{C^k}|\nabla^2\psi_I(t)|_{C^k} + |\nabla \psi_I(t)|_{C^k}^2). 		\label{error estimate of psi_I(t)}
\end{align}

\indent Next we will use the identification of the kernel of $\LL$ and the tangent space of Calabi-Yau structures at $(\til{\omega},\til{J})$, from Lemma \ref{tangent space lemma}, to find a new Calabi-Yau structure denoted $(\omega_I,J_I)$ such that $|\pi_0(\psi_I(t))|_{L^2}$ is small relative to $|\psi_I(t)|_{L^2}$. Furthermore we will show that the new Calabi-Yau structure is contained in a fixed neighborhood of the original Calabi-Yau structure $(\til{\omega},\til{J})$.

\begin{thm} \label{new KE structure} Given $t_0$ and $T > 0$, let $I = [t_0, t_0 + T]$. There exists $\delta(T,\til{g})$ such that if $\sup_{I}|\psi(t)|_{C^k}<\delta$ with $k \geq 2$, then there exists a Calabi-Yau structure $(\omega_I,J_I)$ with the following properties:
\begin{enumerate}
	\item$|\pi_0(\psi_I)|_{L^2(\til{g})}^2\leq \frac{1}{4}|\psi_I|_{L^2(\til{g})}^2$ on $I$
	\item$|(\omega_I - \til{\omega},J_I-\til{J})|_{C^k}\leq C \sup_{I}|\psi|_{C^k}.$
\end{enumerate}

\end{thm}

\proof First by Theorem \ref{Koiso's Theorem} we know that $\mathcal{U}$ has a manifold structure near $(\til{\omega},\til{J})$ and moreover by Lemma \ref{tangent space lemma} we have $\Ker \LL \cong T_{(\til{\omega},\til{J})}\mathcal{U}$.

\indent By identifying $\Ker \LL$ and $T_{(\til{\omega},\til{J})}\mathcal{U}$ we will view $(\til{\omega},\til{J})$ as the origin of $\Ker \LL$. Let $\Phi = \Phi_{(\til{\omega},\til{J})} :\Ker \LL   \rightarrow \mathcal{U}$ denote the exponential map at $(\til{\omega},\til{J})$. Now since $\mathcal{D}_{(\til{\omega},\til{J})}\Phi$ is the identity map, the inverse function theorem may be applied to $\Phi$. By the inverse function theorem there exists a neighborhood $V \subset \Ker \LL$ of $(\til{\omega},\til{J})$ on which the exponential map is invertible.\\
\indent Let $\delta_1$ be small enough so that 
\begin{align}
	|\pi_0(\psi(t_0))|_{C^k} < \delta_1 \text{ implies that } \pi_0(\psi(t_0)) \in V. \label{c1<0 delta1}
\end{align}	
Notice that if $\sup_I|\psi|_{C^k}  < \delta_1$ then since $t_0 \in I$, it is clear that $|\pi_0(\psi(t_0))|_{C^k} < \delta_1$. Hence by the inverse function theorem there is a Calabi-Yau structure $(\omega_I,J_I) \in \mathcal{U}$ such that
\begin{align}
	(\Phi|_{V})^{-1}\big((\omega_I,J_I)\big) = \pi_0(\psi(t_0)). \label{inverse function theorem}
\end{align} 
Applying $\Phi$ to each side of (\ref{inverse function theorem}), it follows from the inverse function theorem that there exists a constant $C$ so that
\begin{align*}
	|(\omega_I-\til{\omega}, J_I-\til{J})|_{C^k} &\leq C |\pi_0(\psi(t_0))|_{C^k}\\
	&\leq C\sup_{I}|\psi|_{C^k}.
\end{align*}
This proves $(2)$.\\
\indent Next, using that $(\Phi|_{V})^{-1}\big((\omega_I,J_I)\big) = \pi_0\big((\omega_I - \til{\omega}, J_I - \til{J})\big)$, from (\ref{inverse function theorem}) we have
\begin{align}
	\pi_0(\psi_I(t_0))=0.\label{kernel}
\end{align}
In other words, there exists a Calabi-Yau structure $(\omega_I,J_I)$ such that at time $t_0$, $\psi_I(t)$ is orthogonal, with respect to $L^2(\til{g})$, to $\Ker \LL$.\\ 
\indent To prove $(1)$ we will carefully study the evolution of $\psi_I(t)$ starting at $t = t_0$ in order to get $L^2$ estimates on $\pi_0(\psi_I)$. First let $||\psi_I||_{M \times I}\dot{=}\int_{I}|\psi_I|_{L^2(\til{g})}$ denote the $L^2$ norm on $M \times I$. Let $\{ B_i \}$ be an orthonormal basis, with respect to $L^2(\til{g})$, of $T_{(\til{\omega},\til{J})}\mathcal{C}$ determined by the eigenspace decomposition of $\LL$. Then there exist constants $c_i$ so that $\{ c_iB_ie^{\lambda_it} \}$ is an orthonormal basis, with respect to $|| \cdot ||_{M \times I}$, of $\Ker \left( \dt - \LL \right) \big|_{M \times I}$ where $\lambda_i$ is the eigenvalue associated to $B_i$.\\ 
\indent We let $\pi^I\big(\psi_I(t)\big)$ denote the projection of $\psi_I(t)$ onto $\Ker \left( \dt - \LL \right)\big|_{M \times I}$. In other words,
\begin{align}
	\dt \pi^I(\psi_I(t))= \LL(\pi^I(\psi_I(t))). \label{evolution of pi^Ipsi_I}
\end{align}
From (\ref{kernel}), we have $\pi^I(\psi_I(t_0))= \sum_{\lambda_i \neq 0}k_i B_i$. It then follows that
\begin{align}
	\pi^I(\psi_I(t)) = \sum_{\lambda_i \neq 0} k_{i}B_{i}e^{\lambda_i(t - t_0)}. \label{pi^Ipsi_I(t)}
\end{align} 
We write
\begin{align}
	\psi_I(t)=\pi^I(\psi_I(t))+\xi_I(t). \label{psi_I(t)2}
\end{align}
\indent Since $\pi^I\big(\psi_I(t)\big)$ is orthogonal to $\Ker \LL$ on $I$, it follows that for $t \in I$,
\begin{align}
	|\pi_0(\psi_I(t))| \leq |\xi_I(t)|. \label{kernel estimate}
\end{align}
Therefore to obtain estimates on $\pi_0\big( \psi_I(t)\big)$ we compute the evolution of $\xi_I(t)$. Moreover, from (\ref{pi^Ipsi_I(t)}) and (\ref{psi_I(t)2}), since $\lambda_i<0$ is bounded away from 0 for all $i$, we have that $\xi_I(t)$ converges exponentially to $\psi_I(t)$. Therefore there is a uniform constant $C$ so that on $I$, 
\begin{align}
	|\xi_I(t)| \leq C|\psi_I(t)|.\label{epsilon bounds}
\end{align} 

\indent To compute the evolution of $\xi_I(t)$ we compare two evolution equations for $\psi_I(t)$. From (\ref{evolution of psi_I(t)}), $\psi_I(t)$ satisfies $\dt\psi_I(t)= \LL(\psi_I(t))+A((\til{\omega},\til{J}),\psi_I)$ and hence,
\begin{align}
	\dt \psi_I(t)=\LL(\pi^I(\psi_I(t)))+\LL(\xi_I(t))+A((\til{\omega},\til{J}),\psi_I(t)). \label{psi_I 1}
\end{align}

Furthermore, $\pi^I(\psi_I(t))$ satisfies (\ref{evolution of pi^Ipsi_I}) and so
\begin{align}
	\dt\psi_I(t) = \dt\big(\pi^I(\psi_I(t))+\xi_I(t)\big) = \LL(\pi^I(\psi_I(t)))+\dt\xi_I(t). \label{psi_I 2}
\end{align}
Combining equations (\ref{psi_I 1}) and (\ref{psi_I 2}) we have that $\xi_I(t)$ evolves by 
\begin{align}
	\dt\xi_I(t) = \LL(\xi_I(t))+A((\til{\omega},\til{J}),\psi_I(t)).\label{evolution of epsilon} 
\end{align}

\indent Recall that $\LL$ is negative semi-definite; and so by (\ref{evolution of epsilon}), on $I$ we have
\begin{align*}
	\dt |\xi_I(t)|_{L^2(\til{g})}^2=2\int_M\bigg\langle \dt\xi_I(t),\xi_I(t)\bigg\rangle_{\til{g}}dvol_{\til{g}} \leq 2\int_MA\left((\til{\omega},\til{J}),\psi_I(t)\right)*\xi_I(t).
\end{align*}
Now using the bounds on $A$ from (\ref{error estimate of psi_I(t)}), the same computation as (\ref{2.14}) shows that 
\begin{align*}
	\dt |\xi_I(t)|_{L^2(\til{g})}^2 \leq C_1\int_M |\nabla^2\psi_I||\psi_I||\xi_I|.
\end{align*}
Hence by (\ref{epsilon bounds}), 
\begin{align}
	\dt |\xi_I(t)|_{L^2(\til{g})}^2 \leq C_2 \int_M |\nabla^2\psi_I||\psi_I|^2. \label{L2 evolution of xi}
\end{align}
\indent Next we assume $\sup_I|\psi(t)|_{C^k}<\delta$ with $k\geq 2$ and $\delta \leq \delta_1$ where $\delta_1$ is from (\ref{c1<0 delta1}). Using part $(2)$ of Theorem \ref{new KE structure} and the triangle inequality, from (\ref{L2 evolution of xi}) it follows that on I
\begin{align*}
	\dt |\xi_I(t)|_{L^2}^2 \leq C_3\delta |\psi_I(t)|_{L^2}^2.
\end{align*}

\indent Now since $\xi_I(t_0) = 0$, 
\begin{align}
	|\xi_I(t)|_{L^2(\til{g})}^2= \int_{t_0}^{t}\frac{\del}{\del s} \big|\xi_I(s)\big|_{L^2(\til{g})}^2ds \leq C_3\delta \int_{t_0}^{t} |\psi_I(s)|_{L^2(\til{g})}^2ds. \label{L2 xi estimate} 
\end{align}
Notice that since $\dt \psi_I(t) = \dt \psi(t)$ is second order in $\psi(t)$ and $\sup_{I}|\psi(t)|_{C^k}<\delta$ with $k \geq 2$, $\dt \psi_I(t)$ is uniformly bounded in terms of $\delta$ and hence each $\psi_I(s)$ for $s \in I$ is uniformly equivalent. Therefore
\begin{align*}
	|\pi_0(\psi_I(t))|_{L^2(\til{g})}^2 \leq |\xi_I(t)|_{L^2(\til{g})}^2 \leq C_4\delta \int_{t_0}^{t}|\psi_I(t)|_{L^2(\til{g})}^2ds = C_5\delta (t - t_0)|\psi_I(t)|_{L^2(\til{g})}^2,
\end{align*}
where the first inequality follows from (\ref{kernel estimate}) and the second is from (\ref{L2 xi estimate}).
To finish the proof we choose $\delta$ small enough so that both $C_5T\delta < \frac{1}{4}$ and $\delta \leq \delta_1$ hold. 
\qed

\medskip
We will now use part $(1)$ of Theorem \ref{new KE structure} to prove $L^2$ exponential decay of $\psi_I(t)$ on $I$. Notice that by (\ref{rho_I leq psi_I}) this implies exponential decay of $\rho_I(t) = (\omega(t)-\omega_I, J(t) - J_I)$ on $I$.

\lemma \label{exponential decay}Let $I$ and $(\omega_I,J_I)$ be as in Theorem \ref{new KE structure}. There exists $\beta>0$ such that if $|\psi|_{C^2}<\beta$ on $I$, then
$$\sup_{\left[t_0+\frac{1}{2}T,t_0+T\right]}\int_M|\psi_I|^2dvol_{\til{g}} \leq e^{-\frac{T\lambda}{2}}\sup_{\left[t_0,t_0+\frac{1}{2}T\right]}\int_M|\psi_I|^2dvol_{\til{g}}$$ 
where $\lambda = min\{ |\lambda_i|: \lambda_i$ is a non-zero eigenvalue of $\LL$ $\} >0$.

\proof 

We compute the evolution of $|\psi_I|_{L^2}^2$ and as in (\ref{evolution of psi_I(t)}) we have
\begin{align}
	\dt \int_M|\psi_I|^2dvol_{\til{g}} = 2\int_M\langle \LL(\psi_I),\psi_I\rangle dvol_{\til{g}}+\int_MA\left((\til{\omega},\til{J}),\psi_I(t)\right)*\psi_Idvol_{\til{g}}. \label{L2 evolution of psi_I}
\end{align}

\indent Recall that by the definition of $\pi_0$, $\psi_I - \pi_0(\psi_I)$ is the component of $\psi_I$ orthogonal to the kernel of $\LL$. Hence by the definition of $\lambda$,
\begin{align}
	2\int_M\langle \LL(\psi_I),\psi_I\rangle dvol_{\til{g}} &\leq -2\lambda|\psi_I-\pi_0(\psi_I)|_{L^2}^2 \label{first term in L2 	psiI1}\\
	&\leq -2\lambda \left(|\psi_I|_{L^2}^2 - |\pi_0(\psi_I)|_{L^2}^2 \right).\label{first term in L2 psiI1'}
\end{align}

\indent Let $\delta$ be the constant from Theorem \ref{new KE structure}. By Theorem \ref{new KE structure} part (1), if $\sup_I|\psi(t)|_{C^2} < \beta_1$ with $\beta_1 \leq \delta$, then from (\ref{first term in L2 psiI1}) and (\ref{first term in L2 psiI1'}) it follows that
\begin{align}
	2\int_M\langle \LL(\psi_I),\psi_I\rangle dvol_{\til{g}} &\leq - \frac{3}{2}\lambda |\psi_I|_{L^2}^2. \label{first term in L2 psiI2}
\end{align}
\indent Next consider the term $\int A * \psi_I$ from (\ref{L2 evolution of psi_I}). We use the estimate on $A$ from (\ref{error estimate of psi_I(t)}) to bound $\int A * \psi$. Notice that if $|\psi_I|_{C^2}<\beta_3$, then as in (\ref{2.15}),
\begin{align}
	\int_MA\left((\til{\omega},\til{J}),\psi_I(t)\right)*\psi_Idvol_{\til{g}} \leq C\beta_3|\psi_I|_{L^2}^2. \label{second term in L2 psiI}
\end{align}
Now we choose $\beta_3$ small enough so that 
\begin{align}
	C\beta_3 < \frac{\lambda}{2}. \label{beta3}
\end{align}
Let $\beta_2$ be sufficiently small so that $|\psi|_{C^2} < \beta_2$ on $I$ implies that $|\psi_I|_{C^2} < \beta_3$ on $I$. Notice that this can be done using the triangle inequality and part $(2)$ of Theorem \ref{new KE structure}.\\
\indent Finally we choose $\beta = min \{ \beta_1, \beta_2 \} $. Combining (\ref{L2 evolution of psi_I}), (\ref{first term in L2 psiI2}), (\ref{second term in L2 psiI}) and (\ref{beta3}) it follows that if $|\psi|_{C^2} < \beta$ on $I$, then

\begin{align}
	\dt \int_M|\psi_I|^2dvol_{\til{g}} \leq -\lambda \int_M|\psi_I|^2dvol_{\til{g}}. \label{L2 psi_I bound}
\end{align}
Integrating from $t_0$ to $t$ gives
\begin{align}
	|\psi_I(t)|_{L^2}^2 \leq e^{-\lambda(t - t_0)}|\psi_I(t_0)|_{L^2}^2. \label{L2 decay of psi_I}
\end{align}
Finally since (\ref{L2 psi_I bound}) implies that $|\psi_I(t)|_{L^2}^2$ is decreasing, plugging $t_0 + \frac{1}{2}T$ into (\ref{L2 decay of psi_I}) proves the lemma.
\qed

\medskip
This gives exponential $L^2$ decay of $\psi_I(t)$ on $I$. Next we prove a general result about parabolic flows (cf. Lemma 8.8 in \cite{HCF}). The following lemma says roughly that exponential decay at a later time implies exponential decay at an earlier time.
\lemma \label{exponential decay later implies exponential decay earlier}There exists $\nu>0$ so that if $\kappa$ solves the parabolic flow equation $\dt \kappa = \LL (\kappa) + A(\kappa)$ and $|\kappa(t)|_{C^k}<\nu$ for all $t \in [0,t_0 + T]$, then
\begin{align}
	\sup_{\left[t_0+\frac{1}{2}T,t_0+T\right]}\int_M|\kappa|^2 \leq e^{-\frac{T\lambda}{2}}\sup_{\left[t_0,t_0+\frac{1}{2}T\right]}\int_M|\kappa|^2 \label{exp decay later}
\end{align}
implies that 
\begin{align}
	\sup_{\left[t_0,t_0+\frac{1}{2}T\right]}\int_M|\kappa|^2 \leq e^{-\frac{T\lambda}{2}}\sup_{\left[t_0-\frac{1}{2}T,t_0\right]}\int_M|\kappa|^2. \label{exp decay earlier}
\end{align}
\proof Suppose, by way of contradiction, that the lemma fails to hold. Then there is a sequence $\nu_i\rightarrow0$ with $\kappa_i(t)$ solving $\dt \kappa_i = \LL (\kappa_i) + A(\kappa_i)$ and $|\kappa_i|_{C^k}<\nu_i$ on $[0,t_0+T]$ and moreover (\ref{exp decay later}) holds but (\ref{exp decay earlier}) does not. Parabolically rescale the solution $\kappa_i$; that is let $\til{\kappa}(t)_i=\nu_i^{-1}\kappa_i(\nu_it)$. Now for all $i$, $|\til{\kappa}_i|_{C^k}<1$ on $[0,\nu_i^{-1}(t_0 + T)] \supset [0, t_0 + T]$ and so by compactness we get a convergent subsequence $\til{\kappa}(t)_i\rightarrow \til{\kappa}(t)_{\infty}$ on $[0,t_0 + T]$ as $i\rightarrow \infty$. Now since $\til{\kappa}_i$ solves $\dt \til{\kappa}_i = \LL (\nu_i\til{\kappa}_i) + A(\nu_i\til{\kappa}_i)$ and $A(\kappa)$ is quadratic in $\kappa$ this implies that $\til{\kappa}_{\infty}(t)$ solves the linear equation
\begin{align}
	\dt \til{\kappa}_{\infty} = \LL(\til{\kappa}_{\infty}). \label{evolution of kappa infinity}
\end{align}
Furthermore since (\ref{exp decay later}) and (\ref{exp decay earlier}) are scale invariant it follows that for $\til{\kappa}_{\infty}$ (\ref{exp decay later}) holds but (\ref{exp decay earlier}) does not. This is a contradiction.\\
\indent To see the contradiction, first notice that (\ref{evolution of kappa infinity}) implies that
\begin{align}
	\dt |\til{\kappa}_{\infty}|_{L^2}^2 \leq 0 \label{L2 kappa infinity dec}
\end{align} 
as $\LL$ is negative semi-definite. It then follows that 
\begin{align}
	|\til{\kappa}_{\infty}(t_0 + \tfrac{1}{2}T)|_{L^2}^2 &= \sup_{\left[t_0+\frac{1}{2}T,t_0+T\right]}|\til{\kappa}_{\infty}|_{L^2}^2  \label{exp decay kappa infinity 1} \\
	&\leq  e^{-\frac{T\lambda}{2}}\sup_{\left[t_0,t_0 + \frac{1}{2}T\right]}|\til{\kappa}_{\infty}|_{L^2}^2\label{exp decay kappa infinity 2}\\
	&= e^{-\frac{T\lambda}{2}}|\til{\kappa}_{\infty}(t_0)|_{L^2}^2,\label{exp decay kappa infinity 3}
\end{align}
where the inequality follows from (\ref{exp decay later}). As above, let $\{ B_i \}$ be an orthonormal basis, with respect to $L^2(\til{g})$, of $T_{(\til{\omega},\til{J})}\mathcal{C}$ determined by the eigenspace decomposition of $\LL$. We can now write $\til{\kappa}_{\infty}(t)$ with respect to this basis,
\begin{align}
	\til{\kappa}_{\infty}(t) = \til{\kappa}_{\infty}(t_0) \left( \sum_i B_ie^{\lambda_i(t - t_0)} \right). \label{kappa infinity (t)}
\end{align}
Notice that at time $t = t_0 + \frac{1}{2}T$ we have
\begin{align}
	|\til{\kappa}_{\infty}(t_0 + \tfrac{1}{2}T)|_{L^2}^2 = |\til{\kappa}_{\infty}(t_0)|_{L^2}^2 \left( \sum_i e^{T\lambda_i} \right). \label{kappa infinity (t0 + 1/2T)}
\end{align}
By combining (\ref{exp decay kappa infinity 1}), (\ref{exp decay kappa infinity 2}), (\ref{exp decay kappa infinity 3}) and (\ref{kappa infinity (t0 + 1/2T)}), it follows that
\begin{align}
	\sum_i e^{T\lambda_i} \leq e^{-\frac{T\lambda }{2}}. \label{exp decay rate}
\end{align}
From (\ref{kappa infinity (t)}) it follows that 
\begin{align}
	|\til{\kappa}_{\infty}(t_0 - \tfrac{1}{2}T)|_{L^2}^2 = |\til{\kappa}_{\infty}(t_0)|_{L^2}^2 \left( \sum_i e^{-T\lambda_i} \right). \label{kappa infinity (t0 - 1/2T)}
\end{align}
Finally using (\ref{kappa infinity (t0 - 1/2T)}) and the concavity of $f(x) = \frac{1}{x}$ we see that
\begin{gather}
\begin{split} \label{last inequality}
	|\til{\kappa}_{\infty}(t_0)|_{L^2}^2 &= |\til{\kappa}_{\infty}(t_0 - \tfrac{1}{2}T)|_{L^2}^2 \frac{1}{\left( \sum_i e^{-T\lambda_i} \right)}\\
	&\leq |\til{\kappa}_{\infty}(t_0 - \tfrac{1}{2}T)|_{L^2}^2 \left( \sum_i e^{T\lambda_i} \right) \\
	&\leq e^{-\frac{T\lambda }{2}}|\til{\kappa}_{\infty}(t_0 - \tfrac{1}{2}T)|_{L^2}^2.
\end{split}
\end{gather}
The last inequality in (\ref{last inequality}) follows from (\ref{exp decay rate}). Notice that the above inequality along with (\ref{L2 kappa infinity dec}) imply that (\ref{exp decay earlier}) holds. This is a contradiction. 
\qed

\medskip
\begin{cor}\label{exp decay cor} Given $T>0$ and $j \geq 1$ there exists $\alpha = \alpha(T,j) > 0$ such that if $|\psi(t)|_{C^2} < \alpha$ on $[0,(j + 1)T]$ then there exists a Calabi-Yau structure $(\omega_j,J_j)$ so that $\rho_j(t) = (\omega(t) - \omega_j,J(t) - J_j)$ satisfies the following exponential decay estimate:
$$|\rho_j(t)|_{L^2}^2 \leq Ce^{-\frac{\lambda t}{2}}$$
for $t \in [0,(j + 1)T]$ and a constant $C$ independent of $j$. 
\end{cor}
\proof First notice that by (\ref{rho_I leq psi_I L2}) it suffices to prove exponential decay of $\psi_j(t)$, the tangent vector associated to $\rho_j(t)$. We will prove exponential decay of $\psi_j(t)$ using the previous two lemmas and Theorem \ref{new KE structure}.\\
\indent Let $\delta$, $\beta$ and $\nu$ be the small constants from Theorem \ref{new KE structure}, Lemma \ref{exponential decay} and Lemma \ref{exponential decay later implies exponential decay earlier} respectively. In order to apply the above lemmas and Theorem \ref{new KE structure} we let $\alpha = min \{ \delta, \beta, \nu \}$ and assume that $|\psi|_{C^2} < \alpha$ on $[0, (J + 1)T]$. Employing Theorem \ref{new KE structure} (with $t_0 = jT$) and Lemma \ref{exponential decay}, there exists a Calabi-Yau structure, denoted $(\omega_j, J_j)$, such that
\begin{align}
	\sup_{\left[\left(j+\frac{1}{2}\right)T,(j + 1)T\right]}|\psi_j|_{L^2}^2 \leq e^{-\frac{T\lambda}{2}}\sup_{\left[jT,\left(j+\frac{1}	{2}\right)T\right]}|\psi_j|_{L^2}^2. \label{cor eqn 1}
\end{align}
\indent Now, Lemma \ref{exponential decay later implies exponential decay earlier} says that exponential decay at a later time implies exponential decay at an earlier time. In particular, from Lemma \ref{exponential decay later implies exponential decay earlier} and (\ref{cor eqn 1}) it follows that for any $k \in \{ \frac{n}{2} : n \in \mathbb{Z} \text{ and } 1 \leq n \leq 2j+1 \}$,
\begin{align}
	\sup_{\left[kT,\left(k + \frac{1}{2}\right)T\right]}|\psi_j|_{L^2}^2 \leq e^{-\frac{T\lambda}{2}}\sup_{\left[\left(k-\frac{1}	{	2}\right)T,kT\right]}|\psi_j|_{L^2}^2. \label{psi_j L2 k estimate}
\end{align}
Applying (\ref{psi_j L2 k estimate}) with $k = \frac{1}{2}$ implies that for any $t \in \left[\frac{1}{2}T,T \right]$,  
\begin{align}
	|\psi_j(t)|_{L^2}^2 \leq e^{-\frac{T\lambda}{2}}\sup_{\left[0,\frac{1}{2}T\right]}|\psi_j|_{L^2}^2 \leq e^{-\frac{\lambda t}{2}}\sup_{\left[0,\frac{1}{2}T\right]}|\psi_j|_{L^2}^2. \label{psi_j L2 k estimate 2}
\end{align}
\indent Next, using (\ref{psi_j L2 k estimate}) with $k = \frac{1}{2}$ and $k = 1$ yields
\begin{align*}
	\sup_{\left[T,\frac{3}{2}T\right]}|\psi_j|_{L^2}^2 \leq e^{-\frac{T\lambda}{2}}\sup_{\left[\frac{1}{2}T,T\right]}|\psi_j|_{L^2}^2 \leq e^{-T\lambda}\sup_{\left[0,\frac{1}{2}T\right]}|\psi_j|_{L^2}^2.
\end{align*}
Therefore it follows that for $t \in \left[T,\frac{3}{2}T \right]$,
\begin{align}
	|\psi_j(t)|_{L^2}^2 \leq e^{-T\lambda} \sup_{\left[0,\frac{1}{2}T\right]}|\psi_j|_{L^2}^2 \leq e^{-\frac{\lambda t}{2}}\sup_{\left[0,\frac{1}{2}T\right]}|\psi_j|_{L^2}^2. \label{psi_j L2 k estimate 3}
\end{align}
Combining (\ref{psi_j L2 k estimate 2}) and (\ref{psi_j L2 k estimate 3}) yields $L^2$ exponential decay of $\psi_j(t)$ on $\left[\frac{1}{2}T, \frac{3}{2}T \right].$ Iterating this argument, we see that for $t \in \left[\frac{1}{2}T, (j + 1)T \right]$,
\begin{align*}
	|\psi_j(t)|_{L^2}^2 \leq e^{-\frac{\lambda t}{2}}\sup_{\left[0,\frac{1}{2}T\right]}|\psi_j|_{L^2}^2 \leq Ce^{-\frac{\lambda t}{2}}.
\end{align*}
Notice that $C$ is independent of $j$. Indeed by assumption $|\psi(t)|_{C^2} < \alpha$ on $[0,(j+1)T]$. Hence part $(2)$ of Theorem \ref{new KE structure} and the triangle inequality imply that $|\psi_j(t)|_{C^2} < C$ on $[0,(j+1)T]$, where $C$ is independent of $j$.
\qed

\medskip
Notice that the decay estimate from Corollary \ref{exp decay cor} may fail to hold for intervals beyond $I_j$. In order to maintain exponential decay we want to choose another Calabi-Yau structure $(\omega_{j+1},J_{j+1})$ to which the solution exponentially converges. To ensure that we can continue this process we need to prove that $|\psi(t)|_{C^2}$ is small for all time so that Corollary \ref{exp decay cor} may be applied on any interval. This is the purpose of the following theorem. As a corollary we will prove the existence of a Calabi-Yau structure, denoted $(\omega_{KE},J_{KE})$, to which the flow exponentially converges.

\begin{thm} \label{long-time existence}
Let $(M, \til{\omega}, \til{J})$ be a closed complex manifold with $(\til{\omega},\til{J})$ a Calabi-Yau structure. Given $\epsilon' > 0$ and $k \geq 0$ there exists $\epsilon > 0$ so that $|\rho(0)|_{C^{\infty}} < \epsilon$ implies that $|\psi(t)|_{C^k} < \epsilon'$ on $[0, \infty)$.
\end{thm}

\begin{proof}
We will employ Theorem \ref{start close, stay close}. To do this we make explicit $T$, $\epsilon$, and $\epsilon'$.
\begin{enumerate}
	\item Let $T$ be large enough so that
		\begin{align*}
			\frac{T^2C_3(k+2)}{e^{T\lambda}-1} + \frac{1}{e^{T\lambda}}<\frac{1}{2}.
		\end{align*}
		Where $C_3(k+2)$ is a constant depending only on $k$ and $(\til{\omega},\til{J})$.
	\item Choose
		\begin{align*} 
			\epsilon'= \alpha.
		\end{align*}
		Where $\alpha$ is the constant from Corollary \ref{exp decay cor}.
	\item Choose $\epsilon$ sufficiently small so that $(\omega(t),J(t))$ exists on $[0,3T]$ and
		\begin{align*}
			\sup_{[0,2T]}|\psi(t)|_{C^k}<\frac{\epsilon'}{e^{T\lambda}}.
		\end{align*}
\end{enumerate}

We want to prove that $|\psi(t)|_{C^k} < \epsilon'$ on $[0, \infty)$. Suppose by way of contradiction there is a finite maximal time $T'$ such that $|\psi(t)|_{C^k} < \epsilon'$ on $[0,T')$ with $k \geq 2$. Write $[0,T') = [0,T] \cup [T,2T] \cup \cdots \cup [NT, T')$ and let $[jT, (j+1)T] = I_j$. Also let $(\omega_j,J_j)$ denote the Calabi-Yau structure, from Corollary \ref{exp decay cor}, to which $(\omega(t),J(t))$ exponentially converges on $I_j$. Using (\ref{cor eqn 1}) and applying Lemma \ref{exponential decay later implies exponential decay earlier} iteratively we have
\begin{align}
	\sup_{I_{j-1} \cup I_j}|\psi_j|^2_{L^2}\leq e^{-\lambda T(j-1)}\sup_{[0,\frac{T}{2}]}|\psi_j|^2_{L^2}.\label{decay}
\end{align}

We again use a parabolic regularity argument to prove that from (\ref{decay}) we can obtain a $C^{k + 2}$ decay estimate.

\lemma \label{parabolic regularity} There exists a constant $\alpha > 0$ so that if both $|\psi|_{C^2} < \alpha$ on $I_j$ and $\sup_{I_{j-1} \cup I_j}|\psi_j|^2_{L^2}\leq e^{-\lambda T(j-1)}\sup_{[0,\frac{T}{2}]}|\psi_j|^2_{L^2}$, then there exists a constant $C(k + 2)$ such that
$$\sup_{I_j}|\psi_j|_{C^{k + 2}} < C(k + 2) Te^{-T\lambda (j - 1) }\sup_{\left[0,\frac{1}{2}T \right]}|\psi_j|_{L^2}^2.$$
\proof
The proof of Lemma \ref{parabolic regularity} uses essentially the same argument as the proof of Lemma \ref{L2 to Ck decay} hence we omit some of the details.\\
\indent From (\ref{L2 evolution and estimate}) we have
\begin{align*}
	\dt |\psi_j|_{L^2}^2 \leq -|\nabla \psi_j|_{L^2}^2 + C_1|\psi_j|_{L^2}^2.
\end{align*}
Fix $t \in I_j$ and integrate from $(j - 1)T$ to $t$;
\begin{gather} \label{parabolic regularity L12 1}
\begin{split}
	\int_{(j - 1)T}^{t}|\nabla \psi_j|_{L^2}^2 &\leq |\psi_j((j-1)T)|_{L^2}^2 + C_1\int_{(j-1)T}^{t}|\psi_j|_{L^2}^2 \\
	&\leq CTe^{-T\lambda (j - 1)} \sup_{\left[0,\frac{1}{2}T \right]}|\psi_j|_{L^2}^2, 
\end{split}
\end{gather}
where the second inequality follows from the $L^2$ exponential decay assumption.\\
\indent Next let $\theta(s)$ be a smooth function which is 0 for $s \leq (j - 1)T$, monotonically increasing from 0 to 1 for $s\in [(j - 1)T,jT]$ and equal to 1 for $s \geq jT$. As in Lemma \ref{parabolic regularity}, $\theta(s)$ will be used to deal with the boundary terms that arise in the estimates below. Now from (\ref{1}) we have 
\begin{align*}
	\dt |\nabla \psi_j|_{L^2}^2 \leq C_5|\psi|_{L^2}^2+ C_6 |\nabla \psi|_{L^2}^2
\end{align*}
and since $\theta(s)$ and its derivative are uniformly bounded, it follows that
\begin{align*}
	\frac{\partial}{\partial s}\left(\theta(s) |\nabla \psi_j(s)|_{L^2}^2\right) \leq C_7|\psi_j|_{L^2}^2+ C_8 |\nabla \psi_j|_{L^2}^2.
\end{align*}
We now integrate from $(j - 1)T$ to $t \in I_j$ and use that $\theta((j-1)T) = 0$ and $\theta(t) = 1$,
\begin{align}
	|\nabla \psi_j (t)|_{L^2}^2 &\leq C_7 \int_{(j-1)T}^{t}|\psi_j|_{L^2}^2 + C_8\int_{(j-1)T}^{t}|\nabla \psi_j|_{L^2}^2 \label{parabolic regularity L12 2} \\
	&\leq  CTe^{-T\lambda (j - 1)} \sup_{\left[0,\frac{1}{2}T \right]}|\psi_j|_{L^2}^2, \label{parabolic regularity L12 3}
\end{align}
where the first and second terms on the right-hand side of (\ref{parabolic regularity L12 2}) were bounded using the $L^2$ exponential decay assumption and (\ref{parabolic regularity L12 1}) respectively. Notice that (\ref{parabolic regularity L12 2}) and (\ref{parabolic regularity L12 3}) yield the desired $L^{1,2}$ exponential decay estimate.\\
\indent Continuing in this way, on $I_j$ we get
\begin{align*}
	|\psi_j (t)|_{L^{p,2}}^2 \leq C(p)Te^{-T\lambda (j - 1)} \sup_{\left[0,\frac{1}{2}T \right]}|\psi_j|_{L^2}^2,
\end{align*}
moreover by the Sobolev Embedding Theorem, for any $t \in I_j$
\begin{align*}
	|\psi_j (t)|_{C^{k+2}} \leq C(k + 2)Te^{-T\lambda (j - 1)} \sup_{\left[0,\frac{1}{2}T \right]}|\psi_j|_{L^2}^2.
\end{align*}
\qed
 
 \medskip
From Lemma \ref{parabolic regularity} it follows that
\begin{align*}
	\sup_{I_j}\bigg|\dt \psi \bigg|_{C^k} \leq C_2(k + 2) Te^{-T\lambda (j - 1) }\epsilon' 
\end{align*}
since $\big|\dt \psi \big|_{C^k} = \big|\dt \psi_j \big|_{C^k} \leq C\sup_{I_j}|\psi_j|_{C^{k+2}}$. Hence, for $j \geq 2$ and any $t\in I_j$, by integrating we see that 
\begin{align*}
	|\psi(t)|_{C^k} &\leq T \sup_{I_j}\bigg| \dt \psi \bigg|_{C^k} + \sup_{I_{j-1}}|\psi|_{C^k} \\
	&\leq T \sum_{l=2}^{j}\sup_{I_l}\bigg| \dt \psi \bigg|_{C^k} + \sup_{I_0 \cup I_1}|\psi|_{C^k} \\
	&< \epsilon' C_3(k + 2)T^2\left( \frac{1}{e^{\lambda T}} +\frac{1}{e^{2\lambda T}}+\dots +\frac{1}{e^{(j-1)\lambda T}} 		\right)+ \sup_{I_0 \cup I_1}|\psi|_{C^k} \\
	&\leq \epsilon' \frac{C_3(k+2)T^2}{e^{\lambda T}-1}+ \sup_{I_0 \cup I_1}|\psi|_{C^k} \\
	&\leq \epsilon' \frac{C_3(k+2)T^2}{e^{\lambda T}-1}+ \frac{\epsilon'}{e^{T\lambda}} \\
	&< \frac{\epsilon'}{2}.
\end{align*}
Where the final inequality is from our choice of $T$ and $\epsilon$. The key observation here is that the above inequality is independent of both $j \geq 2$ and $t \in I_j$. Hence the above inequality holds for $j = N$ which contradicts the maximality of $T'$. Therefore $T'=\infty$.
\end{proof}

The important thing to notice about Theorem \ref{long-time existence} is that it allows us to find a Calabi-Yau structure $(\omega_{KE},J_{KE})$ to which the flow converges.
\cor Under the assumptions of Theorem \ref{dynamic stability} with $(M, \til{\omega},\til{J})$ a Calabi-Yau manifold, there exists a Calabi-Yau structure $(\omega_{KE},J_{KE})$ to which the flow exponentially converges.

\proof By Theorem \ref{long-time existence} we know that given $\epsilon' > 0$ and $k \geq 0$, there exists $\epsilon > 0$ such that if $|\rho(0)|_{C^{\infty}} < \epsilon$, then $|\psi(t)|_{C^k} < \epsilon'$ for all $t \geq 0$. Let $\epsilon' = \alpha$, where $\alpha$ is the small constant from Corollary \ref{exp decay cor}. Recall that $\rho_j(t) = (\omega(t) - \omega_j,J(t) - J_j)$. Now since $|\psi(t)|_{C^k} < \alpha$ for all $t \geq 0$, by Corollary \ref{exp decay cor} there exists a sequence of Calabi-Yau structures $\{ (\omega_j,J_j) \}$ such that $\rho_j(t)$ exponentially decays in $L^2$ for all $t\in[0,(j+1)T]$ and for each $j$. Specifically,
\begin{align}
	|\rho_j(t)|_{L^2} \leq Ce^{-\frac{\lambda t}{2}} \label{rho_j exp decay est}
\end{align}
for $t \in [0,(j+1)T]$ and for each $j$. \\
\indent Next, by part $(2)$ of Theorem \ref{new KE structure}, each of these Calabi-Yau structures $(\omega_j,J_j)$ is contained in a fixed neighborhood of $(\til{\omega},\til{J})$, in particular $|(\omega_j - \til{\omega},J_j-\til{J})|_{C^k}\leq C\epsilon'$. Therefore as $j\rightarrow \infty$ we have a convergent subsequence $(\omega_j,J_j)\rightarrow (\omega_{KE},J_{KE})$. And hence by (\ref{rho_j exp decay est}) we have $L^2$ exponential convergence of $(\omega(t),J(t))$ to $(\omega_{KE},J_{KE})$
for all $t \geq 0$. Finally applying the parabolic regularity argument of Lemma \ref{L2 to Ck decay} it follows that the exponential convergence of $(\omega(t),J(t))$ to $(\omega_{KE},J_{KE})$ is $C^k$ convergence, that is
\begin{align}
	|(\omega(t)-\omega_{KE},J(t) - J_{KE})|_{C^k} \leq Ce^{-\frac{\lambda t}{2}}. \label{rho_KE exp decay est}
\end{align}

\noindent In other words, $(\omega(t),J(t))$ is contained in a ball of radius $Ce^{-\frac{\lambda t}{2}}$ of the limit Calabi-Yau structure $(\omega_{KE},J_{KE})$ for all $t \geq 0$. This gives us the desired exponential decay estimate.

\qed

\bibliographystyle{plain}
\bibliography{biblio}

\end{document}